\documentclass[11pt]{article}
\usepackage[margin=1in]{geometry}

\usepackage{tabmac}
\usepackage{amsfonts, amsmath, amsthm, amssymb, mathrsfs, enumerate,
  xypic, color, paralist}
\usepackage{hyperref}

\input xy
\xyoption{all}

\numberwithin{equation}{section}
\newtheorem{theorem}[equation]{Theorem}
\newtheorem{proposition}[equation]{Proposition}
\newtheorem{lemma}[equation]{Lemma}
\newtheorem{corollary}[equation]{Corollary}

\newtheorem{problem}[equation]{Problem}

\theoremstyle{definition}
\newtheorem{rmk}[equation]{Remark}
\newenvironment{remark}[1][]{%
    \begin{rmk}[#1] \pushQED{\qed}}{\popQED \end{rmk}}

\newtheorem{eg}[equation]{Example}
\newenvironment{example}[1][]{%
    \begin{eg}[#1] \pushQED{\qed}}{\popQED \end{eg}}

\newtheorem{definition}[equation]{Definition}

\numberwithin{equation}{section}

\newcommand{\eps}{\varepsilon}
\renewcommand{\phi}{\varphi}

\newcommand{\CC}{\mathcal{C}}

\newcommand{\FF}{\mathcal{F}}

\newcommand{\LL}{\mathcal{L}}
\newcommand{\OO}{\mathcal{O}}

\newcommand{\RR}{\mathcal{R}}
\renewcommand{\ss}{\mathfrak{S}}

\newcommand{\Z}{\mathbf{Z}}
\newcommand{\Q}{\mathbf{Q}}
\newcommand{\Sym}{\mathrm{Sym}}
\newcommand{\D}{\mathrm{D}}

\newcommand{\Hs}{\mathrm{H}}
\DeclareMathOperator{\gr}{gr}

\newcommand{\GL}{\mathbf{GL}}
\renewcommand{\hom}{\operatorname{Hom}}
\newcommand{\rank}{\operatorname{rank}}

\newcommand{\surj}{\twoheadrightarrow}

\renewcommand{\b}{\mathfrak{b}}
\newcommand{\h}{\mathfrak{h}}
\newcommand{\gl}{\mathfrak{gl}}
\newcommand{\ch}{\operatorname{ch}}
\newcommand{\sch}{\operatorname{sch}}

\newcommand{\Schub}{\mathscr{S}}

\newcommand{\row}{\mathrm{Row}}

\DeclareMathOperator{\depth}{depth}

\DeclareMathOperator{\Spec}{Spec}
\DeclareMathOperator{\ann}{ann}
\DeclareMathOperator{\coker}{coker}
\DeclareMathOperator{\codim}{codim}
\DeclareMathOperator{\supp}{supp}

\DeclareMathOperator{\length}{length}
\newcommand{\Flag}{{\bf Flag}}

\title{Schubert complexes and degeneracy loci}
\author{Steven V Sam}
\date{April 28, 2010}

\begin{document}

\maketitle

\begin{abstract} Given a generic map between flagged vector bundles on
  a Cohen--Macaulay variety, we construct maximal Cohen--Macaulay
  modules with linear resolutions supported on the Schubert-type
  degeneracy loci. The linear resolution is provided by the Schubert
  complex, which is the main tool introduced and studied in this
  paper. These complexes extend the Schubert functors of Kra\'skiewicz
  and Pragacz, and were motivated by the fact that Schur complexes
  resolve maximal Cohen--Macaulay modules supported on determinantal
  varieties. The resulting formula in K-theory provides a ``linear
  approximation'' of the structure sheaf of the degeneracy locus,
  which can be used to recover a formula due to Fulton.
\end{abstract}


\section*{Introduction.}

Let $X$ be an equidimensional Cohen--Macaulay (e.g., nonsingular)
variety, and let $\phi \colon E \to F$ be a map of vector bundles over
$X$, with ranks $e$ and $f$ respectively. Given a number $k \le
\min(e,f)$, let $D_k(\phi)$ be the degeneracy locus of points $x$
where the rank of $\phi$ restricted to the fiber of $x$ is at most
$k$. Then $\codim D_k(\phi) \le (e-k)(f-k)$, and in the case of
equality, the Thom--Porteous formula expresses the homology class of
$D_k(\phi)$ as an evaluation of a multi-Schur function at the Chern
classes of $E$ and $F$ (see \cite[\S 3.5.4]{manivel}). Also in the
case of equality, the Schur complex associated with the rectangular
partition $(f-k) \times (e-k)$ (see \cite{ABW} or \cite[\S
2.4]{weyman} for more about Schur complexes) of $\phi$ is a linear
locally free resolution for a Cohen--Macaulay coherent sheaf whose
support is $D_k(\phi)$. This resolution gives a formula in the
K-theory of $X$. In the case that $X$ is smooth, there is an
isomorphism from an associated graded of the K-theory of $X$ to the
Chow ring of $X$ (see Section~\ref{section:ktheory} for more
details). Then the image of this complex recovers the Thom--Porteous
formula, and the complex provides a ``linear approximation'' of the
syzygies of $D_k(\phi)$.

The situation was generalized by Fulton \cite{flags} as follows. We
provide the additional data of a flag of subbundles $E_\bullet$ for
$E$ and a flag of quotient bundles $F_\bullet$ for $F$, and we can
define degeneracy loci for an array of numbers which specifies the
ranks of the restriction maps $E_p \to F_q$. The rank functions that
give rise to irreducible degeneracy loci are indexed by permutations
in a natural way. Under the right codimension assumptions, one can
express the homology class of a given degeneracy locus as a
substitution of a double Schubert polynomial with the Chern classes of
the quotients $E_i / E_{i-1}$ and the kernels $\ker(F_j \to
F_{j-1})$. The motivation for this work was to complete the analogy of
this situation with the previous one by constructing ``Schubert
complexes'' which would be acyclic whenever the degeneracy loci has
the right codimension.

Building on the constructions for Schubert functors by Kra\'skiewicz
and Pragacz of \cite{schubertfunctor}, we construct these complexes
over an arbitrary (commutative) ring $R$ from the data of two free
$R$-modules $M_0$, $M_1$, with given flags of submodules,
respectively, quotient modules, and a map $\partial \colon M_0 \to
M_1$. We can also extend the construction to an arbitrary scheme. We
show that they are acyclic when a certain ideal defined in terms of
minors of $\partial$ has the right depth, i.e., they are
``depth-sensitive.''  Our main result is that in the situation of
Fulton's theorem, the complex is acyclic and the Euler characteristic
provides the formula in the same sense as above. Our proof uses
techniques from commutative algebra, algebraic geometry, and
combinatorics. Again, the complexes are linear and provide a ``linear
approximation'' to the syzygies of Fulton's degeneracy loci.  As a
special case of Fulton's degeneracy loci, one gets Schubert varieties
inside of (type A) partial flag varieties.

Using the work of Fomin, Greene, Reiner, and Shimozono \cite{fomin},
we construct explicit bases for the Schubert complex in the case that
$M_0$ and $M_1$ are free. This basis naturally extends their notion of
balanced labelings and the generating function of the basis elements
gives what seems to be a new combinatorial expression for double
Schubert polynomials. Furthermore, the complex naturally affords a
representation of the Lie superalgebra of upper triangular matrices
(with respect to the given flags) in $\hom(M_0, M_1)$, and its graded
character is the double Schubert polynomial.

The article is structured as follows. In Section~\ref{bslsection} we
recall some facts about double Schubert polynomials and balanced
labelings. We introduce balanced super labelings (BSLs) and prove some
of their properties. In Section~\ref{linearalgebrasection} we extend
the construction for Schubert functors to the $\Z/2$-graded setting
and show that they have a basis naturally indexed by the BSLs. In
Section~\ref{complexsection} we construct the Schubert complex from
this $\Z/2$-graded Schubert functor. Using some facts about the
geometry of flag varieties, we show that the acyclicity of these
complexes is controlled by the depth of a Schubert determinantal
ideal. In the case of acyclicity and when the coefficient ring is
Cohen--Macaulay, we show that the cokernel of the complex is a
Cohen--Macaulay module which is generically a line bundle on its
support. We also give some examples of Schubert complexes. Finally, in
Section~\ref{locisection}, we relate the acyclicity of the Schubert
complexes to a degeneracy locus formula of Fulton. We finish with some
remarks and possible future directions.

\subsection*{Conventions.} 

The letter $K$ is reserved for a field of arbitrary characteristic. If
$X$ is a scheme, then $\OO_X$ denotes the structure sheaf of
$X$. Throughout, all schemes are assumed to be separated. A variety
means a reduced scheme which is of finite type over $K$. We treat the
notions of locally free sheaves and vector bundles as the same, and
points will always refer to closed points. The fiber of a vector
bundle $E$ at a point $x \in X$ is denoted $E(x)$ and refers to the
stalk $E_x$ tensored with the residue field $k(x)$.  Given a line
bundle $L$ on $X$, $c_1(L)$ denotes the first Chern class of $L$,
which we think of as a degree $-1$ endomorphism of the Chow groups
${\rm A}_*(X)$. For an element $\alpha \in {\rm A}_*(X)$, and an
endomorphism $c$ of ${\rm A}_*(X)$, we will use the notation $c \cap
\alpha$ to denote $c$ applied to $\alpha$.

\subsection*{Acknowledgements.}

The author thanks Jerzy Weyman for helpful discussions and for
pointing out the connection between Schur complexes and the
Thom--Porteous formula which was the starting point for this work, and
also for reading previous drafts of this paper. We also thank an
anonymous referee for some comments on a previous draft. The author
was supported by an NSF graduate research fellowship and an NDSEG
fellowship while this work was done.

\section{Double Schubert polynomials.} \label{bslsection}

\subsection{Preliminaries.} \label{section:preliminaries}

Let $\Sigma_n$ be the permutation group on the set
$\{1,\dots,n\}$. Since we are thinking of $\Sigma_n$ as a group of
functions, we will multiply them as functions, e.g., if $s_1$ and
$s_2$ are the transpositions that switch 1 and 2, and 2 and 3,
respectively, then $s_1s_2$ is the permutation $1 \mapsto 2$, $2
\mapsto 3$, $3 \mapsto 1$. We will use inline notation for
permutations, so that $w$ is written as $w(1) w(2) \cdots
w(n)$. Proofs for the following statements about $\Sigma_n$ can be
found in \cite[\S 2.1]{manivel}. Let $s_i$ denote the transposition
which switches $i$ and $i+1$. Then $\Sigma_n$ is generated by
$\{s_1,\dots,s_{n-1}\}$, and for $w \in \Sigma_n$, we define the {\bf
  length} of $w$ to be the least number $\ell(w)$ such that $w =
s_{i_1} \cdots s_{i_{\ell(w)}}$. Such a minimal expression is a {\bf
  reduced decomposition} for $w$. All reduced expressions can be
obtained from one another using only the {\bf braid relations}:
$s_is_j = s_js_i$ for $|i-j|>1$ and $s_is_{i+1}s_i =
s_{i+1}s_is_{i+1}$. We can also write $\ell(w) = \# \{i<j \mid w(i) >
w(j)\}$. The {\bf long word} $w_0$ is the unique word with maximal
length, and is defined by $w_0(i) = n+1-i$.

We will use two partial orders on $\Sigma_n$. The {\bf (left) weak
  Bruhat order}, denoted by $u \le_W w$, holds if some reduced
decomposition of $u$ is the suffix of some reduced decomposition of
$w$.\footnote{In \cite{manivel}, the weak Bruhat order is defined in
  terms of prefixes. We point out that these two definitions are
  distinct, but this will not cause any problems.} We denote the {\bf
  strong Bruhat order} by $u \le w$, which holds if some reduced
decomposition of $w$ contains a subword that is a reduced
decomposition of $u$. It follows from the definition that $u \le w$ if
and only if $u^{-1} \le w^{-1}$. For a permutation $w$, let $r_w(p,q)
= \#\{i \le p \mid w(i) \le q \}$ be its {\bf rank function}. Then $u
\le w$ if and only if $r_u(p,q) \ge r_w(p,q)$ for all $p$ and $q$ (the
inequality on rank functions is reversed).

Given a polynomial (with arbitrary coefficient ring) in the variables
$\{x_i\}_{i \ge 1}$, let $\partial_i$ be the {\bf divided difference
  operator}
\begin{align} \label{divideddiff} (\partial_iP)(x_1, x_2, \dots) =
  \frac{P(\dots, x_{i-1}, x_i, x_{i+1}, \dots) - P(\dots, x_{i-1},
    x_{i+1}, x_i, \dots)}{x_i - x_{i+1}}.
\end{align}
The operators $\partial_i$ satisfy the braid relations:
$\partial_i\partial_j = \partial_j\partial_i$ when $|i-j| > 1$ and
$\partial_i \partial_{i+1} \partial_i
= \partial_{i+1} \partial_i \partial_{i+1}$.

For the long word $w_0 \in \Sigma_n$, set $\ss_{w_0}(x,y) = \prod_{i +
  j \le n} (x_i - y_j)$. In general, if $\ell(ws_i) = \ell(w) - 1$, we
set $\ss_{ws_i}(x,y) = \partial_i \ss_w(x,y)$, where we interpret
$\ss_w(x,y)$ as a polynomial in the variables $\{x_i\}_{i \ge 1}$ with
coefficients in the ring $\Z[y_1,y_2,\dots]$. These polynomials are
the {\bf double Schubert polynomials}, and are well-defined since the
$\partial_i$ satisfy the braid relations and the braid relations
connect all reduced decompositions of a permutation. The definition of
these polynomials is due to Lascoux and Sch\"utzenberger
\cite{ls}. They enjoy the following stability property: if we embed
$\Sigma_n$ into $\Sigma_{n+m}$ by identifying permutations of
$\Sigma_n$ with permutations of $\Sigma_{n+m}$ which pointwise fix
$\{n+1, n+2, \dots, n+m\}$, then the polynomial $\ss_w(x,y)$ is the
same whether we regard $w$ as an element of $\Sigma_n$ or
$\Sigma_{n+m}$ \cite[Corollary 2.4.5]{manivel}.

Define the {\bf single Schubert polynomials} by $\ss_w(x) =
\ss_w(x,0)$. We will use the identity \cite[Proposition
2.4.7]{manivel}
\begin{align} \label{induction} \ss_w(x,y) = \sum_{u \le_W w} \ss_u(x)
  \ss_{uw^{-1}}(-y).
\end{align}

\subsection{Balanced super labelings.}

For the rest of this article, we fix a totally ordered alphabet
$\cdots < 3' < 2' < 1' < 1 < 2 < 3 < \cdots$. The elements $i'$ will
be referred to as {\bf marked} and the elements $i$ will be referred
to as {\bf unmarked}. 

For a permutation $w$, define its {\bf diagram} $D(w) = \{(i, w(j))
\mid i < j,\ w(i) > w(j) \}$. Note that $\# D(w) = \ell(w)$. Our
convention is that the box $(i,j)$ means row number $i$ going from top
to bottom, column number $j$ going from left to right, just as with
matrix indexing. An alternative way to get the diagram of $D(w)$ is as
follows: for each $i$, remove all boxes to the right of $(i,w(i))$ in
the same row and all boxes below $(i,w(i))$ in the same column
including $(i,w(i))$. The complement is $D(w)$. See
Figure~\ref{figure:permdiagram} for an example with $w = 35142$. Here
the boxes $(i,w(i))$ are marked with $\bullet$ and the other removed
boxes are marked with $\times$.
\begin{figure}
\[
  {\tableau[scY]{, , \bullet, \times, \times | , , \times, , \bullet |
    \bullet, \times, \times, \times, \times | \times, , \times,
    \bullet, \times | \times, \bullet, \times, \times, \times }}
\]
\caption{$D(35142)$}
\label{figure:permdiagram}
\end{figure}

Let $T$ be a labeling of $D(w)$. The {\bf hook} of a box $b \in D(w)$
is the set of boxes in the same column below it, and the set of boxes
in the same row to the right of it (including itself). A hook is {\bf
  balanced} (with respect to $T$) if it satisfies the following
property: when the entries are rearranged so that they are weakly
increasing going from the top right end to the bottom left end, the
label in the corner stays the same. A labeling is {\bf balanced} if
all of the hooks are balanced. Call a labeling $T$ of $D(w)$ with
entries in our alphabet a {\bf balanced super labeling (BSL)} if it is
balanced, column-strict (no repetitions in any column) with respect to
the unmarked alphabet, row-strict with respect to the marked alphabet,
and satisfies $j' \le T(i,j) \le i$ for all $i$ and $j$ (this last
condition will be referred to as the {\bf flag conditions}). To be
consistent with the identity permutation, we say that an empty diagram
has exactly one labeling.

\begin{example} We list the BSL for some long words.

$n=3$, $\ss_{321}(x,y) = (x_1-y_1)(x_1-y_2)(x_2-y_1)$
\[ { \tableau[scY]{1,1|2} } \quad { \tableau[scY]{1,1|1'}} \quad {
  \tableau[scY]{1,2'|2}} \quad { \tableau[scY]{1',2'|1}} \quad {
  \tableau[scY]{1,1'|2}} \quad { \tableau[scY]{1',1|1'}} \quad {
  \tableau[scY]{1',2'|2}} \quad { \tableau[scY]{1',2'|1'}}
\]

$n=4$, $\ss_{4321}(x,y) = (x_1-y_1) (x_1-y_2)(x_2-y_1)
(x_1-y_3)(x_2-y_2)(x_3-y_1)$.
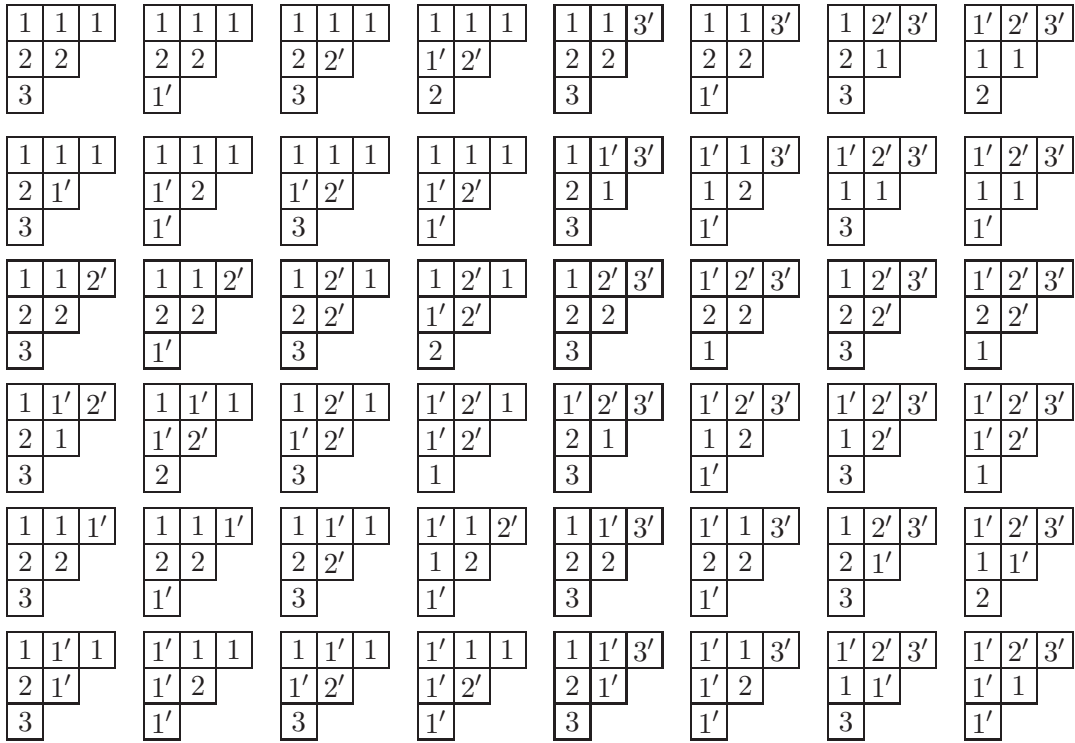
\[
\begin{array}{llllllll}
  { \tableau[scY]{1,1,1 | 2,2 | 3} } \quad 
  { \tableau[scY]{1,1,1 | 2,2 | 1'}} \quad 
  { \tableau[scY]{1,1,1|2,2'|3}} \quad 
  { \tableau[scY]{1,1,1|1',2'|2}} \quad 
  { \tableau[scY]{1,1,3'|2,2|3}} \quad
  { \tableau[scY]{1,1,3'|2,2|1'}} \quad
  { \tableau[scY]{1,2',3'|2,1|3}} \quad
  { \tableau[scY]{ 1',2',3'|1,1|2}}
\end{array}
\]
\[
\begin{array}{llllllll}
  { \tableau[scY]{1,1,1|2,1'|3}} \quad 
  { \tableau[scY]{1,1,1|1',2|1'}} \quad 
  { \tableau[scY]{1,1,1|1', 2'|3}} \quad
  { \tableau[scY]{1,1,1|1',2'|1'}} \quad 
  { \tableau[scY]{1,1',3'|2,1|3}} \quad 
  { \tableau[scY]{ 1',1,3'|1,2|1'}} \quad
  { \tableau[scY]{ 1',2',3'|1,1|3}} \quad
  { \tableau[scY]{ 1',2',3'|1,1|1'}}
\end{array}
\]
\[
\begin{array}{llllllll}
  { \tableau[scY]{1,1,2'|2,2|3}} \quad 
  { \tableau[scY]{1,1,2'|2,2|1'}} \quad 
  { \tableau[scY]{1,2',1|2,2'|3}} \quad
  { \tableau[scY]{1,2',1|1',2'|2}} \quad
  { \tableau[scY]{1,2',3'|2,2|3}} \quad
  { \tableau[scY]{1',2',3'|2,2|1}} \quad
  { \tableau[scY]{1,2',3'|2,2'|3}} \quad
  { \tableau[scY]{ 1',2',3'|2,2'|1}}
\end{array}
\]
\[
\begin{array}{llllllll}
  { \tableau[scY]{1,1',2'|2,1|3}} \quad 
  { \tableau[scY]{1,1',1|1',2'|2}} \quad
  { \tableau[scY]{1,2',1|1',2'|3}} \quad
  { \tableau[scY]{ 1',2',1|1',2'|1}} \quad 
  { \tableau[scY]{ 1',2',3'|2,1|3}} \quad 
  { \tableau[scY]{ 1',2',3'|1,2|1'}} \quad
  { \tableau[scY]{ 1',2',3'|1,2'|3}} \quad
  { \tableau[scY]{ 1',2',3'|1',2'|1}}
\end{array}
\]
\[
\begin{array}{llllllll}
  { \tableau[scY]{1,1,1'|2,2|3}} \quad 
  { \tableau[scY]{1,1,1'|2,2|1'}} \quad 
  { \tableau[scY]{1,1',1|2,2'|3}} \quad 
  { \tableau[scY]{ 1',1,2'|1,2|1'}} \quad
  { \tableau[scY]{1,1',3'|2,2|3}} \quad
  { \tableau[scY]{ 1',1,3'|2,2|1'}} \quad 
  { \tableau[scY]{1,2',3'|2,1'|3}} \quad 
  { \tableau[scY]{ 1',2',3'|1,1'|2}}
\end{array}
\]
\[
\begin{array}{llllllll}
  { \tableau[scY]{1,1',1|2,1'|3}} \quad
  { \tableau[scY]{ 1',1,1|1',2|1'}} \quad
  { \tableau[scY]{1,1',1|1',2'|3}} \quad
  { \tableau[scY]{ 1',1,1|1',2'|1'}} \quad
  { \tableau[scY]{1,1',3'|2,1'|3}} \quad
  { \tableau[scY]{ 1',1,3'|1',2|1'}} \quad
  { \tableau[scY]{ 1',2',3'|1,1'|3}} \quad
  { \tableau[scY]{ 1',2',3'|1',1|1'}}
\end{array}
\]
\[
\begin{array}{llllllll}
  { \tableau[scY]{1,1',2'|2,2|3}} \quad
  { \tableau[scY]{ 1',1,2'|2,2|1'}} \quad
  { \tableau[scY]{1,1',2'|2,2'|3}} \quad
  { \tableau[scY]{ 1',2',1|1',2'|2}} \quad
  { \tableau[scY]{ 1',2',3'|2,2|3}} \quad
  { \tableau[scY]{ 1',2',3'|2,2|1'}} \quad
  { \tableau[scY]{ 1',2',3'|2,2'|3}} \quad
  { \tableau[scY]{ 1',2',3'|1',2'|2}}
\end{array}
\]
\[
\begin{array}{llllllll}
  { \tableau[scY]{1,1',2'|2,1'|3}} \quad
  { \tableau[scY]{ 1',1,2'|1',2|1'}} \quad
  { \tableau[scY]{ 1',2',1|1',2'|3}} \quad
  { \tableau[scY]{ 1',2',1|1',2'|1'}} \quad
  { \tableau[scY]{ 1',2',3'|2,1'|3}} \quad
  { \tableau[scY]{ 1',2',3'|1',2|1'}} \quad
  { \tableau[scY]{ 1',2',3'|1',2'|3}} \quad
  { \tableau[scY]{1',2',3' | 1',2' | 1'}}
\end{array} \qedhere
\]
\end{example}

Let $A = (a_{i,j})$ be an $n \times n$ array. We define left and right
actions of $\Sigma_n$ on $A$ as follows. For $w \in \Sigma_n$, set
$(wA)_{i,j} = A_{i,w(j)}$, and $(Aw)_{i,j} =
A_{w^{-1}(i),j}$. Equivalently, $Aw = (w^{-1}A^t)^t$ where $t$ denotes
transpose. In particular, if $A = D(w)$ is the diagram of a
permutation, and $\ell(wu) = \ell(w) + \ell(u)$, then $D(w)u \subseteq
D(wu)$. It is enough to check this when $u = s_i$ is a
transposition. In this case, the condition $\ell(ws_i) = \ell(w) + 1$
means that $w(i) < w(i+1)$, and then $D(ws_i) = D(w)s_i \cup
\{(i,w(i))\}$. Similarly, $wD(u) \subseteq D(wu)$.

If $w$ is a permutation, then $(i,j) \in D(w)$ is a {\bf border cell}
if $w(i+1) = j$. In particular, if $(i,j)$ is a border cell, then
$w(i) > w(i+1)$, so $(D(w) \setminus (i,j))s_i = D(ws_i)$.

\begin{lemma} \label{removal} Let $T$ be a labeling of $D(w)$ with
  largest label $M$.
  \begin{compactenum}[\rm (a)]
  \item \label{parta} Suppose $(i,j)$ is a border cell which contains
    $M$. Then $T$ is balanced if and only if $(T \setminus (i,j))s_i$
    is balanced.
  \item \label{partb} Suppose $T$ is a BSL and $M$ is unmarked. Then
    every row which contains $M$ must contain an $M$ in a border cell.
  \end{compactenum}
\end{lemma}

\begin{proof} See \cite[Theorem 4.8]{fomin} for \eqref{parta}. Part
  \eqref{partb} follows from \cite[Lemma 4.7]{fomin} \end{proof}

By convention, a BSL of $D(w)$ is an $n \times n$ array which is 0
outside of $D(w)$ and takes values in our alphabet otherwise. We use
the convention that $0+i = i+0 = i$ and $0+i' = i'+0 = i'$ whenever
$i,i'$ is in our alphabet, and also that $1' < 0 < 1$.

\begin{lemma} \label{factorlemma} Let $u$ and $v$ be two permutations
  such that $\ell(uv) = \ell(u) + \ell(v)$. Let $T_u$ be a BSL of
  $D(u)$ using only marked letters, and let $T_v$ be a BSL of $D(v)$
  using only unmarked letters. Then $T = T_uv + uT_v$ is a BSL for
  $D(uv)$, and all BSLs of $w = uv$ come from such a
  ``factorization'' in a unique way.
\end{lemma}

\begin{proof} The condition $j' \le T(i,j) \le i$ is automatic since
  we assumed that $T_u$ contains only marked letters and $T_v$
  contains only unmarked letters. Similarly, the respective
  column-strict and row-strict conditions are automatic. So it is
  enough to check that $T$ is balanced. 

  By Lemma~\ref{removal}, we can factor $v = s_{i_1} s_{i_2} \cdots
  s_{i_{\ell(v)}}$ into simple transpositions such that if we write
  $v_j = s_{i_1} \cdots s_{i_{j-1}} s_{i_j}$, then $T_{v_{\ell(v)}} =
  T_v$, and for $j<\ell(v)$, $T_{v_j}$ is the result of removing a
  border cell with the largest label $L_j$ from $T_{v_{j+1}}$ and
  hence is a balanced labeling. In particular, $L_1 \le L_2 \le
  \cdots \le L_{\ell(v)}$. Set $T_0 = T_u$ and $T_j = T_uv_j +
  uT_{v_j}$ for $1 \le j \le \ell(v)$. Then for $1 \le j \le \ell(v)$,
  $T_j$ is the result of switching rows $i_j$ and $i_{j+1}$ in
  $T_{j-1}$ and replacing the newly made 0 with $L_j$. Since all
  letters in $T_0$ are marked, and $L_1 \le L_2 \le \cdots \le
  L_{\ell(v)}$, we conclude from Lemma~\ref{removal}\eqref{parta} that
  each $T_j$ is balanced, and hence $T = T_{\ell(v)}$ is balanced.

  The last statement also follows from Lemma~\ref{removal}: given a
  BSL of $D(w)$, we can successively remove border cells containing
  the largest labels (which are unmarked), and the result will be a
  BSL of a diagram $D(u)$ for some permutation $u$ which contains only
  unmarked letters. The removals give the desired permutation $v =
  u^{-1}w$.

  For uniqueness, note that if at any point we have two choices of
  border cells to remove in rows $i$ and $j$, then $|i-j| >
  1$. Otherwise, if $j = i+1$, for example, then by the balanced
  condition at the hook of box $(i,w(i+2))$, $T(i,w(i+1)) =
  T(i,w(i+2)) = T(i+1,w(i+2))$, which contradicts our strictness
  conditions. Since $s_i$ and $s_j$ commute for $|i-j|>1$, it does not
  matter which one we do first.
\end{proof}

Given a BSL $T$ of $D(w)$, let $f_T(i)$, respectively $f_T(i')$, be
the number of occurrences of $i$, respectively $i'$. Define a monomial
\begin{align}
  m(T) = x_1^{f_T(1)} \cdots x_{n-1}^{f_T(n-1)} (-y_1)^{f_T(1')}
  \cdots (-y_{n-1})^{f_T((n-1)')}.
\end{align}

One more bit of notation: given a labeling $T$ of $D(w)$, let $T^*$
denote the labeling of $D(w^{-1})$ obtained by transposing $T$ and
performing the swap $i \leftrightarrow i'$.

\begin{theorem} For every permutation $w$,
  \begin{align*}
    \ss_w(x,y) = \sum_T m(T),
  \end{align*}
  where the sum is over all BSL $T$ of $D(w)$.
\end{theorem}

\begin{proof} Suppose we are given a BSL $T$ of $D(w)$. By
  Lemma~\ref{factorlemma}, there exists a unique pair of permutations
  $v^{-1}$ and $u$ such that $v^{-1}u = w$, $\ell(w) = \ell(v^{-1}) +
  \ell(u)$, a BSL $T_{v^{-1}}$ of $D(v^{-1})$ which only uses marked
  letters, and a BSL $T_u$ of $D(u)$ which only uses unmarked letters,
  such that $T = T_uv^{-1} + uT_{v^{-1}}$. The labeling $T_v =
  T^*_{v^{-1}}$ gives a BSL of $D(v)$ which only uses unmarked
  letters.

  Finally, using \eqref{induction} coupled with the fact that
  $\ss_u(x) = \sum_T m(T)$, where the sum is over all BSL of $D(u)$
  using only unmarked letters \cite[Theorem 6.2]{fomin}, we get the
  desired result.
\end{proof}

\begin{remark} The operation $T \mapsto T^*$ gives a concrete
  realization of the symmetry $\ss_w(-y,-x) = \ss_{w^{-1}}(x,y)$
  \cite[Corollary 2.4.2]{manivel}.
\end{remark}

\section{Double Schubert functors.} \label{linearalgebrasection}

\subsection{Super linear algebra preliminaries.}

Let $V = V_0 \oplus V_1$ be a free $\Z/2$-graded module over a
(commutative) ring $R$ with $V_0 = \langle e_1, \dots, e_n \rangle$
and $V_1 = \langle e'_1, \dots, e'_m \rangle$, and let $\gl(m|n) =
\gl(V)$ be the Lie superalgebra of endomorphisms of $V$. Let $\b(m|n)
\subset \gl(m|n)$ be the standard Borel subalgebra of upper triangular
matrices with respect to the ordered basis $\langle e'_m, \dots, e'_1,
e_1, \dots, e_n \rangle$. We will mainly deal with the case $m=n$, in
which case we write $\b(n) = \b(n|n)$, and if it is clear from
context, we will drop the $n$ and simply write $\b$. Also, let
$\b(n)_0 = \gl(V)_0 \cap \b(n)$ be the even degree elements in
$\b(n)$, and again, we will usually denote this by $\b_0$. We also
write $\h(n) \subset \b(n)$ for the Cartan subalgebra of diagonal
matrices (this is a Lie algebra concentrated in degree 0). Let
$\eps'_n, \dots, \eps'_1, \eps_1, \dots, \eps_n$ be the dual basis
vectors to the standard basis of $\h(n)$. For notation, write $(a_n,
\dots, a_1 | b_1, \dots, b_n)$ for $\sum_{i=1}^n (a_i\eps'_i +
b_i\eps_i)$. The even and odd roots of $\b(n)$ are $\Phi_0 = \{\eps'_j
- \eps'_i,\ \eps_i - \eps_j \mid 1 \le i < j \le n\}$ and $\Phi_1 =
\{\eps'_i - \eps_j \mid 1 \le i,j \le n \}$, respectively. The even
and odd simple roots are $\Delta_0 = \{\eps'_{i+1} - \eps'_i,\ \eps_i
- \eps_{i+1} \mid i = 1,\dots,n-1 \}$ and $\Delta_1 = \{\eps'_1 -
\eps_1 \}$.

Given a highest weight representation $W$ of $\b(n)$, we have a weight
decomposition $W = \bigoplus_\lambda W_\lambda$ as a representation of
$\h(n)$. Let $\Lambda$ be the highest weight of $W$. Then every weight
$\lambda$ appearing in the weight decomposition can be written in the
form $\Lambda - \sum n_\alpha \alpha$ where $\alpha$ ranges over the
simple roots of $\b(n)$ and $n_\alpha \in \Z_{\ge 0}$. For such a
$\lambda$, set $\omega(\lambda) = (-1)^{\sum n_\alpha \deg
  \alpha}$. Then we define the {\bf character} and {\bf
  supercharacter} of $W$ as
\begin{align}
\ch W = \sum_\lambda (\dim W_\lambda) e^\lambda, \quad \sch W =
\sum_\lambda \omega(\lambda) (\dim W_\lambda) e^\lambda.
\end{align}
Here the $e^\lambda$ are formal symbols with the multiplication rule
$e^\lambda e^\mu = e^{\lambda + \mu}$.

We will need $\Z/2$-graded analogues of the divided and exterior
powers (see \cite[\S 2.4]{weyman} for the dual versions of our
definitions). Let $F = F_0 \oplus F_1$ be a free $R$-supermodule. Let
$\D$ denote the divided power functor, let $\bigwedge$ denote the
exterior power functor, and let $\Sym$ denote the symmetric power
functor. Then $\bigwedge^i F$ and $\D^i F$ are $\Z$-graded modules
with terms given by
\begin{align}
(\bigwedge^i F)_d = \bigwedge^{i-d} F_0 \otimes \Sym^d F_1, \quad
(\D^i F)_d = \D^{i-d} F_0 \otimes \bigwedge^d F_1.
\end{align}
We can define a coassociative $\Z$-graded comultiplication $\Delta
\colon \D^{i+j} F \to \D^i F \otimes \D^j F$ as follows. On
degree $d$, pick $0 \le a \le i$ and $0 \le b \le j$ such that $a+b =
d$. Then we have the composition $\Delta_{a,b}$
\begin{equation}
  \begin{split}
    (\D^{i+j} F)_d &= \D^{i+j-a-b} F_0 \otimes \bigwedge^{a+b} F_1\\
    & \xrightarrow{\Delta' \otimes \Delta'} \D^{i-a} F_0 \otimes
    \D^{j-b} F_0 \otimes \bigwedge^a F_1 \otimes \bigwedge^b F_1\\
    & \cong \D^{i-a} F_0 \otimes \bigwedge^a F_1 \otimes \D^{j-b}
    F_0 \otimes \bigwedge^b F_1 = (\D^i F)_a \otimes (\D^j F)_b,
  \end{split}
\end{equation}
where $\Delta'$ is the usual comultiplication, and we define $\Delta$
on the degree $d$ part to be $\sum_{a+b = d} \Delta_{a,b}$.

Similarly, we can define an associative $\Z$-graded multiplication $m
\colon \bigwedge^i F \otimes \bigwedge^j F \to \bigwedge^{i+j} F$ as
follows. For degrees $a$ and $b$, we have
\begin{equation}
  \begin{split}
    (\bigwedge^i F)_a \otimes (\bigwedge^j F)_b &= \bigwedge^{i-a} F_0
    \otimes \Sym^a F_1 \otimes \bigwedge^{j-b} F_0 \otimes \Sym^b F_1 \\
    &\cong \bigwedge^{i-a} F_0 \otimes \bigwedge^{j-b} F_0 \otimes
    \Sym^a F_1 \otimes \Sym^b F_1\\
    &\xrightarrow{m' \otimes m'} \bigwedge^{i+j-a-b} F_0 \otimes
    \Sym^{a+b} F_1 = (\bigwedge^{i+j} F)_{a+b},
  \end{split}
\end{equation}
where $m'$ is the usual multiplication.

\subsection{Constructions.}

Define a flag of $\Z/2$-graded submodules
\begin{align} \label{Z2flag} 
  V^\bullet : V^{-n} \subset \cdots \subset V^{-1} \subset V^1 \subset
  \cdots \subset V^n
\end{align}
such that $V^{-1}$ consists of all of the odd elements of $V^n$. We
will say that the flag is {\bf split} if each term and each quotient
is a free module. Fix a permutation $w \in \Sigma_n$. Let $r_k =
r_k(w)$, respectively $c_j = c_j(w)$, be the number of boxes in the
$k$th row, respectively $j$th column, of $D(w)$. Define $\chi_{k,j}$
to be 1 if $(k,j) \in D(w)$ and 0 otherwise. Consider the map
\begin{equation} \label{schubertimage}
  \begin{split}
    \bigotimes_{k=1}^{n-1} \D^{r_k} V^k &\xrightarrow{\otimes
      \Delta} \bigotimes_{k=1}^{n-1} \bigotimes_{j=1}^{n-1}
    \D^{\chi_{k, j}} V^k \cong \bigotimes_{j=1}^{n-1}
    \bigotimes_{k=1}^{n-1} \D^{\chi_{k, j}} V^k\\ 
    & \xrightarrow{\otimes m} \bigotimes_{j = 1}^{n-1} \bigwedge^{c_j}
    V^{w^{-1}(j)} \xrightarrow{\otimes \pi} \bigotimes_{j = 1}^{n-1}
    \bigwedge^{c_j} (V^{w^{-1}(j)} / V^{-j-1}),
  \end{split}
\end{equation}
where $\otimes \pi$ denotes the product of projection maps. Note that
$\D^1 V^k = V^k$ and $\D^0 V^k = R$, so that the multiplication above
makes sense. Then its image $\Schub_w(V^\bullet)$ is the {\bf
  $\Z/2$-graded Schubert functor}, or {\bf double Schubert
  functor}. By convention, the empty tensor product is $R$, so that if
$w$ is the identity permutation, then $\Schub_w(V^\bullet) = R$.

This definition is clearly functorial: given an even map of flags $f
\colon V^\bullet \to W^\bullet$, i.e., $f(V^k) \subset W^k$ for $-n
\le k \le n$, we have an induced map $f \colon \Schub_w(V^\bullet) \to
\Schub_w(W^\bullet)$.

We will focus on the case when $V^{-i} = \langle e'_n, e'_{n-1},
\dots, e'_i \rangle$ and $V^i = V^{-1} + \langle e_1, e_2, \dots, e_i
\rangle$, so that $\Schub_w = \Schub_w(V^\bullet)$ is a
$\b(n)$-module.

\begin{remark} One could dually define the double Schubert functor as
  the image of (dual) exterior powers mapping to symmetric powers. One
  has to be careful, because the $\Z/2$-graded version of exterior
  powers are not self-dual. For the dual of our definition, one uses
  $(\bigwedge^i F)_d = \bigwedge^{i-d} F_0 \otimes \D^d F_1$.  We have
  chosen our definitions to be consistent with
  \cite{schubertfunctor}. This will be especially convenient for
  Theorem~\ref{transitionfiltration}.
\end{remark}

\begin{remark}
  We could also define $\Schub_D(V^\bullet)$ for an arbitrary diagram
  $D$ which does not necessarily come from a permutation. This is
  relevant in \cite[\S 4]{schubertfunctor}, whose proof we use in
  Theorem~\ref{transitionfiltration}. However, since the details will
  go through without significant changes, we will have no need to
  elaborate on this point.
\end{remark}

\begin{lemma} \label{lemma:tensoridentity} Let $w \in \Sigma_n$ and $v
  \in \Sigma_m$ be two permutations. Define a new permutation $u \in
  \Sigma_{n+m}$ by $u(i) = w(i)$ for $i=1,\dots,n$ and $u(n+j) = v(j)$
  for $j=1,\dots,m$. Also, define a permutation $v' \in \Sigma_{n+m}$
  by $v'(i) = i$ for $i=1,\dots,n$ and $v'(n+j) = v(j)$ for
  $j=1,\dots,m$. Then 
  \[
  \Schub_u(V^\bullet) \cong \Schub_w(V^\bullet) \otimes
  \Schub_{v'}(V^\bullet).
  \]
\end{lemma}

\begin{proof} This follows from the definition of double Schubert
  functors, the fact that $D(u) = D(w) \cup D(v')$, and the fact that
  no two cells of $D(w)$ and $D(v')$ lie in the same row or column. 
\end{proof}

\begin{example} Consider $n=3$ and $w = 321$. Then $r_1 = 2$, $r_2 = 1$,
  $c_1 = 2$, and $c_2 = 1$. We need to calculate the image of the map
  \[
  \D^2 V^1 \otimes V^2 \xrightarrow{\Delta \otimes 1} (V^1 \otimes
  V^1) \otimes V^2 \xrightarrow{t_{2,3}} (V^1 \otimes V^2) \otimes V^1
  \xrightarrow{m \otimes 1} \bigwedge^2(V^3 / V^{-2}) \otimes
  V^2/V^{-3},
  \]
  where $t_{2,3}$ is the map that switches the second and third parts
  of the tensor product. Write $x$, $y$ for $e_1$, $e_2$, and $x'$,
  $y'$ for $e'_1$, $e'_2$. We can ignore $e_3$ and $e'_3$ since they
  will not appear in the image. We can write $\D^2 V^1 = \langle x^2,
  x \otimes x', x \otimes y', x' \wedge y' \rangle$ and $V^2 = \langle
  x, y, x', y' \rangle$. Then we have
  \begin{align*}
    m(t_{2,3}(\Delta(x^2 \otimes x))) &= m(t_{2,3}(x \otimes x
    \otimes x)) = 0\\
    m(t_{2,3}(\Delta(x^2 \otimes y))) &= m(t_{2,3}(x \otimes x
    \otimes y)) = (x \wedge y) \otimes x \\
    m(t_{2,3}(\Delta(x^2 \otimes x'))) &= m(t_{2,3}(x \otimes x
    \otimes x')) = (x \otimes x') \otimes x\\
    m(t_{2,3}(\Delta(x^2 \otimes y'))) &= m(t_{2,3}( x \otimes x
    \otimes y')) = 0
  \end{align*}
  \begin{align*}
    m(t_{2,3}(\Delta(x \otimes x' \otimes x))) &= m(t_{2,3}( (x
    \otimes x' + x' \otimes x) \otimes x)) = (x \otimes x')
    \otimes x\\
    m(t_{2,3}(\Delta(x \otimes x' \otimes y))) &= m(t_{2,3}( (x
    \otimes x' + x' \otimes x) \otimes y)) = (x \wedge y) \otimes x' +
    (y \otimes x') \otimes x    \\
    m(t_{2,3}(\Delta(x \otimes x' \otimes x'))) &= m(t_{2,3}( (x
    \otimes x' + x' \otimes x) \otimes x')) = (x \otimes x')
    \otimes x' + x'^2 \otimes x \\
    m(t_{2,3}(\Delta(x \otimes x' \otimes y'))) &= m(t_{2,3}( (x
    \otimes x' + x'    \otimes x) \otimes y')) = 0 
  \end{align*}
  \begin{align*}
    m(t_{2,3}(\Delta(x \otimes y' \otimes x))) &= m(t_{2,3}( (x
    \otimes    y' + y'    \otimes x) \otimes x)) = 0 \\
    m(t_{2,3}(\Delta(x \otimes y' \otimes y))) &= m(t_{2,3}( (x
    \otimes    y' + y'    \otimes x) \otimes  y)) = (x \wedge y) \otimes y'\\
    m(t_{2,3}(\Delta(x \otimes y' \otimes x'))) &= m(t_{2,3}( (x
    \otimes y' + y' \otimes x) \otimes x')) = (x \otimes x')
    \otimes y' \\
    m(t_{2,3}(\Delta(x \otimes y' \otimes y'))) &= m(t_{2,3}( (x
    \otimes y' + y'    \otimes x) \otimes y')) = 0 
  \end{align*}
  \begin{align*}
    m(t_{2,3}(\Delta(x' \wedge y' \otimes x))) &= m(t_{2,3}( (x'
    \otimes y' - y' \otimes x') \otimes x)) = (x \otimes x')
    \otimes y' \\
    m(t_{2,3}(\Delta(x' \wedge y' \otimes y))) &= m(t_{2,3}( (x'
    \otimes y' - y' \otimes x') \otimes y)) = (y \otimes x')
    \otimes y' \\
    m(t_{2,3}(\Delta(x' \wedge y' \otimes x'))) &= m(t_{2,3}( (x'
    \otimes y' - y'    \otimes x') \otimes  x')) = x'^2 \otimes y' \\
    m(t_{2,3}(\Delta(x' \wedge y' \otimes y'))) &= m(t_{2,3}( (x'
    \otimes y' - y' \otimes x') \otimes y')) = 0 \qedhere
  \end{align*}
\end{example}

Here is a combinatorial description of the map
\eqref{schubertimage}. The elements of $\bigotimes_{k=1}^{n-1}
\D^{r_k} V^k$ can be thought of as labelings of $D = D(w)$ such that
in row $k$, only the labels $n', (n-1)', \dots, 1', 1, \dots, k$ are
used, such that there is at most one use of $i'$ in a given row, and
such that the entries in each row are ordered in the usual way (i.e.,
$n' < (n-1)' < \cdots < 1' < 1 < \cdots < k$). Let $\Sigma_D$ be the
permutation group of $D$. We say that $\sigma \in \Sigma_D$ is {\bf
  row-preserving} if each box and its image under $\sigma$ are in the
same row. Denote the set of row-preserving permutations as
$\row(D)$. Let $T$ be a labeling of $D$ that is row-strict with
respect to the marked letters. Let $\row(D)_T$ be the subgroup of
$\row(D)$ that leaves $T$ fixed, and let $\row(D)^T$ be the set of
cosets $\row(D) / \row(D)_T$. Given $\sigma \in \row(D)^T$, and
considering the boxes as ordered from left to right, let $\alpha(T,
\sigma)_k$ be the number of inversions of $\sigma$ among the {\it
  marked letters} in the $k$th row, and define $\alpha(T, \sigma) =
\sum_{k=1}^{n-1} \alpha(T, \sigma)_k$. Note that this number is
independent of the representative chosen since $T$ is row strict with
respect to the marked letters.  Then the comultiplication sends $T$ to
$\sum_{\sigma \in \row(D)^T} (-1)^{\alpha(T, \sigma)} \sigma T$ where
$\sigma T$ is the result of permuting the labels of $T$ according to
$\sigma$.

For the multiplication map, we can interpret the columns as being
alternating in the {\it unmarked letters} and symmetric in the {\it
  marked letters}. We write $m(T)$ for the image of $T$ under this
equivalence relation. Therefore, the map \eqref{schubertimage} can be
defined as
\begin{align} \label{labelingmap}
  T \mapsto \sum_{\sigma \in \row(D)^T} (-1)^{\alpha(T, \alpha)}
  m(\sigma T).
\end{align}

\subsection{A basis and a filtration.} \label{section:filtration}

In order to prove properties of $\Schub_w$, we will construct a
filtration by submodules, which is based on the filtration of the
single Schubert functors introduced by Kra\'skiewicz and Pragacz
\cite{schubertfunctor}.

Let $w \in \Sigma_n$ be a nonidentity permutation. Consider the set of
pairs $(\alpha,\beta)$ such that $\alpha < \beta$ and $w(\alpha) >
w(\beta)$. Choose $(\alpha,\beta)$ to be maximal with respect to the
lexicographic ordering. Let $\gamma_1 < \cdots < \gamma_k$ be the
numbers such that $\gamma_t < \alpha$ and $w(\gamma_t) < w(\beta)$,
and such that $\gamma_t < i < \alpha$ implies that $w(i) \notin
\{w(\gamma_t), w(\gamma_t) + 1, \dots, w(\beta)\}$. Then we have the
following identity of double Schubert polynomials
\begin{align} \label{maxtransition} \ss_w = \ss_v \cdot (x_\alpha -
  y_{w(\beta)}) + \sum_{t=1}^k \ss_{\psi_t},
\end{align}
where $v = wt_{\alpha,\beta}$ and $\psi_t = w t_{\alpha,\beta}
t_{\gamma_t, \alpha}$. Here $t_{i,j}$ denotes the transposition which
switches $i$ and $j$. See, for example, \cite[Exercise
2.7.3]{manivel}. The formula in \eqref{maxtransition} will be called a
{\bf maximal transition} for $w$. Define the {\bf index} of a
permutation $u$ to be the number $\sum_k (k-1) \#\{j > k \mid u(k) >
u(j) \}$. Note that the index of $\psi_t$ is smaller than the index of
$w$.

When $w = s_i$ is a simple transposition, $v = 1$ is the identity, $k
= 1$, and $\psi_1 = s_{i-1}$. See
Example~\ref{eg:simpletranspositions} for more details regarding the
filtration in this case.

\begin{theorem} \label{transitionfiltration} Let $V^\bullet$ be a
  split flag as in \eqref{Z2flag}. Given a nonidentity permutation $w
  \in \Sigma_n$, let \eqref{maxtransition} be the maximal transition
  for $w$. Then there exists a functorial $\b$-equivariant filtration
  \[
  0 = F_0 \subset F_1 \subset \cdots \subset F_k \subset F' \subset F
  = \Schub_w(V^\bullet)
  \]
  such that $F / F' \cong \Schub_v(V^\bullet) \otimes V^\alpha / V^{\alpha-1}$,
  $F' / F_k \cong \Schub_v(V^\bullet) \otimes V^{-w(\beta)} /
  V^{-w(\beta)-1}$, and $F_t / F_{t-1} \cong \Schub_{\psi_t}(V^\bullet)$
  for $t=1,\dots,k$.
\end{theorem}

\begin{proof} For notation, write $W^i = V^i / V^{-w(i) - 1}$, and let
  $p \colon W^\beta \to V^\beta / V^{\alpha-1}$ be the projection
  map. Define the $\b$-equivariant morphism $\phi'$ by the composition
  \begin{align*}
    &\bigwedge^{c_{w(1)}} W^1 \otimes \cdots \otimes
    \bigwedge^{c_{w(\alpha)}} W^\alpha \otimes \cdots \otimes
    \bigwedge^{c_{w(\beta)}} W^\beta \otimes \cdots \xrightarrow{1
      \otimes \cdots \otimes 1 \otimes \cdots \otimes \Delta \otimes
      \cdots}\\
    &\bigwedge^{c_{w(1)}} W^1 \otimes \cdots \otimes 
    \bigwedge^{c_{w(\alpha)}} W^\alpha \otimes \cdots \otimes
    \bigwedge^{c_{w(\beta)}-1} W^\beta \otimes W^\beta \otimes \cdots
    \xrightarrow{T} \\
    &\left( \bigwedge^{c_{w(1)}} W^1 \otimes \cdots \otimes
      \bigwedge^{c_{w(\beta)}-1} W^\beta \otimes \cdots \otimes
      \bigwedge^{c_{w(\alpha)}} W^\alpha \otimes \cdots \right)
    \otimes W^\beta
  \end{align*}
  where $T$ is the map which switches the order of the tensor product
  in the way prescribed. Let $\phi = (1 \otimes p) \circ \phi'$. We
  set $F' = \ker \phi$. Let $C_w$ and $C_v$ be the $\b$-cyclic
  generators of $\Schub_w(V^\bullet)$ and $\Schub_v(V^\bullet)$,
  respectively. These are given by BSLs where the $i$th row only has
  the label $i$. By maximality of the pair $(\alpha, \beta)$ (with
  respect to the property $\alpha < \beta$ and $w(\alpha) >
  w(\beta)$), the lowest box in column $w(\beta)$ is in row
  $\alpha$. Hence, restricting $\phi$ to $F = \Schub_w(V^\bullet)$, we
  get $\phi(C_w) = C_v \otimes e_\alpha$. This gives an isomorphism
  $F/F' \to \Schub_v(V^\bullet) \otimes V^\alpha / V^{\alpha-1}$.

  Let $X \in \b$ be the matrix defined by $X(e_\alpha) =
  e'_{w(\beta)}$ and $X(e_i) = 0$ for $i \ne \alpha$ and $X(e'_j) = 0$
  for all $j$. We claim that the rightmost box in row $\alpha$ has
  column index $w(\beta)$. If not, then there is a box $(\alpha,
  w(\beta')) \in D(w)$ with $w(\beta') > w(\beta)$ and $\beta' >
  \alpha$. If $\beta' < \beta$, then $(\beta', \beta) > (\alpha,
  \beta)$ which contradicts the maximality of $(\alpha,
  \beta)$. Otherwise, if $\beta' > \beta$, we have $(\alpha, \beta') >
  (\alpha, \beta)$ which also contradicts maximality. This
  contradiction proves the claim. The claim implies that
  $\phi'(X(C_w)) = C_v \otimes e'_{w(\beta)}$, and hence $X(C_w) \in
  \ker \phi$. Letting $F''$ be the $\b$-submodule generated by
  $X(C_w)$, we get an isomorphism $F'' / \ker \phi' \cong
  \Schub_v(V^\bullet) \otimes V^{-w(\beta)} / V^{-w(\beta)-1}$.

  Using the notation of \cite[\S 4]{schubertfunctor} with the obvious
  changes (see also \cite[Remark 5.3]{schubertfunctor}), let $F_t =
  \sum_{r \le t} S_{\mathcal{I}_r}(V^\bullet)$.\footnote{There is a
    typo in the definition of $\mathcal{F}_t$ in
    \cite{schubertfunctor} regarding $\le$ versus $<$.} The proofs
  from \cite[\S 4]{schubertfunctor} of the fact that there is a
  surjection $F_t / F_{t-1} \to \Schub'_{\psi_t}(V^\bullet)$ in the
  ungraded case extend to the $\Z/2$-graded case. We just need to show
  that these surjections are actually isomorphisms and that $F' =
  F''$. Since this is all defined over $\Z$ and obtained for arbitrary
  $R$ via extension of scalars, it is enough to prove the
  corresponding statements when $R$ is a field of arbitrary
  characteristic.

  We will use the proof of Step 3 in \cite[\S 4]{schubertfunctor}. The
  key steps there involve a tensor product identity \cite[Lemma
  1.8]{schubertfunctor}, using the maximal transitions, and verifying
  the theorem for the simple transpositions. The tensor product
  identity in our case is Lemma~\ref{lemma:tensoridentity}, and the
  maximal transitions still exist. Also, the fact that the statement
  is valid for simple transpositions can be seen directly, or see
  Example~\ref{eg:simpletranspositions}. The rest of the proof goes
  through using the definitions $d_w = \dim_R \Schub_w(V^\bullet)$ and
  $z_w = \ss_w(1, -1)$. So we have defined the desired filtration, and
  the functoriality is evident from the constructions.
\end{proof}

Given any labeling $T$, denote its weight by $w(T) = (a_{-n}, \dots,
a_{-1} | a_1, \dots, a_n)$, where $a_i$ is the number of times that
the label $i$ is used, and $a_{-i}$ is the number of times that the
label $i'$ is used. We define a {\bf dominance order} $\ge$ by
$(a_{-n}, \dots, a_{-1} | a_1, \dots, a_n) \ge (a'_{-n}, \dots,
a'_{-1} | a'_1, \dots, a'_n)$ if $\sum_{i=-n}^k a_i \ge \sum_{i=-n}^k
a'_i$ for all $-n \le k \le n$.

\begin{theorem} \label{basis} Assume that the flag $V^\bullet$ is
  split. The images of the BSLs under \eqref{schubertimage} form a
  basis over $R$ for $\Schub_w$.
\end{theorem}

\begin{proof} Since the BSLs are defined when $R = \Z$, and are
  compatible with extension of scalars, it is enough to show that the
  statement is true when $R = K$ is an infinite field of arbitrary
  characteristic, so we will work in this case.
 
  We can show linear independence of the BSLs following the proof of
  \cite[Theorem 7.2]{fomin}. Combined with
  Theorem~\ref{transitionfiltration}, this will show that they form a
  basis. First, we note that $\Schub_w$ has a weight decomposition
  since it is a highest weight module of $\b(n)$, so we only need to
  show linear independence of the BSLs in each weight space. Second,
  if $T$ is column-strict, row-strict, and satisfies the flagged
  conditions, then $T \ne \pm \sigma T$ in $\Schub_w$ whenever $\sigma
  \in \row(D)_T$. This implies that if $T$ is a BSL, then the image of
  $T$ under \eqref{schubertimage} is nonzero.

  First write $L = \ell(w)$. We assign to $T$ a reduced decomposition
  $s_{i_1}s_{i_2} \cdots s_{i_{L}}$ of $w$ following the method in
  Lemma~\ref{factorlemma} by induction on $L$. If $T$ contains
  unmarked letters, let $M$ be the largest such label, and let $i_L$
  denote the smallest row index which contains $M$ in a border
  cell. Let $s_{i_1} \cdots s_{i_{L - 1}}$ be the reduced
  decomposition assigned to the labeling $T \setminus (i_L, w(i_L))$
  of $ws_{i_{L}}$, so that we get the reduced decomposition $s_{i_1}
  \cdots s_{i_{L}}$ for $w$. If $T$ contains no unmarked letters, let
  $s_{i^*_1} \cdots s_{i^*_{L}}$ be the reduced decomposition
  associated to the labeling $T^*$ of $w^{-1}$ and then assign the
  reduced decomposition $s_{i^*_L} \cdots s_{i^*_1}$ to $T$. So we can
  write this reduced decomposition as $s^*(T)s(T)$ where $s^*(T)$,
  respectively $s(T)$, corresponds to the transpositions coming from
  removing marked, respectively unmarked, letters. We will totally
  order reduced decompositions as follows: $s_{i_1} \cdots s_{i_{L}} <
  s_{i'_1} \cdots s_{i'_{L}}$ if there exists a $j$ such that $i_j <
  i'_j$ and $i_k = i'_k$ for $j+1 \le k \le L$. We say that
  $s^*(T')s(T') \le s^*(T)s(T)$ if either $s(T') < s(T)$ (the ordering
  for reduced decompositions), or $s(T') = s(T)$ and $s^*(T')^{-1} <
  s(T)^{-1}$ (the inverse means write the decomposition backwards).

  Taking into account the description \eqref{labelingmap}, we show
  that if $m(T') = \pm m(\sigma T)$ where $T'$ and $T$ are BSLs and
  $\sigma \in \row(D)$, then $s^*(T')s(T') \le s^*(T)s(T)$. Note that
  since we assume that $T$ and $T'$ have the same weights, we have
  $s^*(T')s(T') = s^*(T)s(T)$ if and only if $T = T'$. The BSLs are
  linearly independent in $\bigotimes_k \D^{r_k} V^k$, so by induction
  on $\le$, we see that the coefficients of any linear dependence of
  their images in $\Schub_w$ must all be zero.

  So suppose that $m(T') = \pm m(\sigma T)$ holds and choose
  representatives $T'$ and $\sigma T$ that realize this
  equality. First suppose that $T$ contains an unmarked letter. Then
  so does $T'$, and let $M$ be the largest such one. Write $s(T) =
  s_{i_r} \cdots s_{i_L}$ and $s(T') = s_{i'_r} \cdots
  s_{i'_L}$. Since $m$ only affects entries within the same column,
  the $M$ in row $i'_L$ is moved to some row with index $\le i'_L$
  because $M$ occupies a border cell. By definition of $i_L$, all
  instances of $M$ in $\sigma T$ lie in rows with index $\ge i_L$
  since $\sigma$ is row-preserving. Hence the equality $m(T') =
  m(\sigma T)$ implies that $i'_L \le i_L$.

  If $i'_L < i_L$, there is nothing left to do. So suppose that $i'_L
  = i_L$. Then $M$ lies in the same border cell $b$ in both $T$ and
  $T'$. Hence $T' \setminus b = \sigma (T \setminus b)$, and we
  conclude by induction. So we only need to handle the case that $T$
  (and hence $T'$) do not contain any unmarked letters. In this case,
  we pass to $T^*$ and $T'^*$, and the above shows that $s^*(T')^{-1}
  \le s^*(T)^{-1}$, so we are done.
\end{proof}

The above proof does not establish how one can write the image of an
arbitrary labeling as a linear combination of the images of the
BSLs. Such a straightening algorithm is preferred, but we have not
been successful in finding one, so we leave this task as an open
problem.

\begin{problem} Find an algorithm for writing the image of an
  arbitrary labeling of $D(w)$ as a linear combination of the images
  of the BSLs of $D(w)$.
\end{problem}

\begin{corollary} Identify $x_i = -e^{\eps_i}$ and $y_i =
  -e^{\eps'_i}$ for $1 \le i \le n$. Then
  \begin{align*}
  \ch \Schub_w = \ss_w(-x,y), \quad \sch \Schub_w = \ss_w(x,y).
  \end{align*}
\end{corollary}

\begin{corollary} \label{filtration} Choose an ordering of the set of
  permutations below $w$ in the weak Bruhat order: $1 = v_1 \prec v_2
  \prec \cdots \prec v_N = w$ such that $v_i \prec v_{i+1}$ implies
  that $\ell(v_i) \le \ell(v_{i+1})$. Then there exists a
  $\b$-equivariant filtration
  \[
  0 = F_0 \subset F_1 \subset \cdots \subset F_N = \Schub_w
  \]
  such that 
  \[
  F_i / F_{i-1} \cong \Schub'_{v_i} \otimes \Schub''_{wv_i^{-1}}
  \]
  as $\b_0$-modules.
\end{corollary}

\begin{proof} 
  Let $\Schub'$, respectively $\Schub''$, denote the usual Schubert
  functor which uses only unmarked, respectively marked, letters. Let
  $W_v = \Schub'_v \otimes \Schub''_{wv^{-1}}$. Theorem~\ref{basis}
  implies that we have a $\b_0$-equivariant decomposition $\Schub_w =
  \bigoplus_{v \le_W w} W_v$. Let $F_i = \bigoplus_{j \le i}
  W_{v_j}$. Then $F_i$ is a $\b_0$-submodule, and applying an element
  of $\b \setminus \b_0$ to $W_{v_j}$ can only give elements in
  $W_{v_k}$ where $\ell(v_k) < \ell(v_j)$. So $F_i$ is in fact a
  $\b$-submodule, and we have the desired filtration.
\end{proof}

\section{Schubert complexes.} \label{complexsection}

Now we can use the above machinery to define Schubert complexes. We
start with the data of two split flags $F^\bullet_0 : 0 = F^0_0
\subset F^1_0 \subset \cdots \subset F^n_0 = F_0$ and $F^\bullet_1 :
F^{-n}_1 \subset F^{-n+1}_1 \subset \cdots \subset F^{-1}_1 = F_1$,
and a map $\partial \colon F_0 \to F_1$ between them. Given the flag
for $F_0$, we pick an ordered basis $\{ e_1, \dots, e_n \}$ for it
such that $e_i \in F_0^i \setminus F_0^{i-1}$. Similarly, we pick an
ordered basis $\{ e'_1, \dots, e'_n \}$ for $F_1$ such that $e'_i \in
F_1^{-i} \setminus F_1^{-i-1}$. Given these bases, we can represent
$\partial$ as a matrix. This matrix representation will be relevant
for the definition of certain ideals later.

Equivalently, we can give $F^\bullet_1$ as a quotient flag $F_1 = G^n
\surj G^{n-1} \surj \cdots \surj G^1 \surj G^0 = 0$, so that the
correspondence is given by $F^{-i}_1 = \ker(G^n \surj G^{i-1})$. Note
that $F_1^{-i} / F_1^{-i-1} = \ker(G^i \surj G^{i-1})$. We assume that
each quotient has rank 1. Then we form a flagged $\Z/2$-graded module
$F$ with even part $F_0$ and odd part $F_1$. The formation of divided
and exterior products commutes with the differential $\partial$ by
functoriality, so we can form the {\bf Schubert complex} $\Schub_w(F)$
for a permutation $w \in \Sigma_n$.

\begin{proposition} \label{complexterms} The $i$th term of
  $\Schub_w(F)$ has a natural filtration whose associated graded is
  \[
  \bigoplus_{\substack{v \le_W w \\ \ell(v) = i }} \Schub_v(F_0)
  \otimes \Schub_{wv^{-1}}(F_1).
  \]
\end{proposition}

\begin{proof} This is a consequence of
  Corollary~\ref{filtration}. \end{proof}

\begin{proposition} \label{complexfiltration} Let $\partial \colon F_0
  \to F_1$ be a map. With the notation as in
  Theorem~\ref{transitionfiltration}, there is a functorial
  $\b$-equivariant filtration of complexes
  \[
  0 = C_0 \subset C_1 \subset \cdots \subset C_k \subset C' \subset C
  = \Schub_w(\partial)
  \]
  such that $C / C' \cong \Schub_v(\partial)[-1] \otimes F_0^\alpha /
  F_0^{\alpha-1}$, $C' / C_k \cong \Schub_v(\partial) \otimes
  F_1^{-w(\beta)} / F_1^{-w(\beta)+1}$, and $C_t / C_{t-1} \cong
  \Schub_{\psi_t}(\partial)$ for $t=1,\dots,k$.
\end{proposition}

\begin{proof} The filtration of Theorem~\ref{transitionfiltration}
  respects the differentials since everything is defined in terms of
  multilinear operations. The grading shift of $C/C'$ follows from the
  fact that the $F_0$ terms have homological degree 1.
\end{proof}

\begin{corollary} Let $\partial \colon F_0 \to F_1$ be a flagged
  isomorphism. Then $\Schub_w(\partial)$ is an exact complex whenever $w
  \ne 1$.
\end{corollary}

\begin{proof} This is an immediate consequence of
  Proposition~\ref{complexfiltration} using induction on length and
  index, and the long exact sequence on homology: when $w = s_i$,
  exactness is obvious.
\end{proof}

\subsection{Flag varieties and K-theory.} \label{section:ktheory}

Throughout this section, we use \cite{intersection} as a
reference. The reader may wish to see \cite[Appendix B.1,
B.2]{intersection} for the conventions used there.

We will need some facts about the geometry of flag varieties. Let
$V$ be a vector space with ordered basis $\{e_1, \dots, e_n\}$. Then
the complete flag variety $\Flag(V)$ can be identified with $\GL(V)/B$
where $B$ is the subgroup of upper triangular matrices with respect to
the given basis. For a permutation $w \in \Sigma_n$, we define the
{\bf Schubert cell} $\Omega_w$ to be the $B$-orbit of the flag
\[
\langle e_{w(1)} \rangle \subset \langle e_{w(1)}, e_{w(2)} \rangle
\subset \cdots \subset \langle e_{w(1)}, \dots, e_{w(n-1)} \rangle
\subset V.
\]
Then $\Omega_w$ is an affine space of dimension $\ell(w)$ (see
\cite[\S 3.6]{manivel}), and the flag variety is a disjoint union of
the $\Omega_w$. The {\bf Schubert variety} $X_w$ is the closure of
$\Omega_w$. Alternatively,
\[
X_w = \{ W_\bullet \in \Flag(V) \mid \dim(W_p \cap \langle e_1, \dots,
e_q \rangle) \ge r_w(p,q) \}.
\]

Recall from Section~\ref{section:preliminaries} that $r_w(p,q) = \#\{i
\le p \mid w(i) \le q\}$. Given a matrix $\partial$ and a permutation
$w$, let $I_w(\partial)$ be the ideal generated by the $(r_w(p,q)+1)
\times (r_w(p,q)+1)$ minors of the upper left $p \times q$ submatrix
of $\partial$. It is clear that $I_v \subseteq I_w$ if and only if $v
\le w$. In the case that $\partial$ is a generic matrix of variables
over some coefficient ring $R$, let $X(w)$ be the variety defined by
$I_w(\partial) \subset R[\partial_{i,j}]$. We refer to the ideals
$I_w(\partial)$ as {\bf Schubert determinantal ideals}, and the
varieties $X(w)$ as {\bf matrix Schubert varieties}.  Given a
permutation $w$, we say that a cell $\alpha$ in the diagram $D(w)$ is
a {\bf southeast corner} if the cells to the immediate right of
$\alpha$ and immediately below $\alpha$ do not belong to $D(w)$.

\begin{theorem} \label{theorem:matrixschubert}
  Let $\partial$ be a generic matrix defined over a field, and let $w$
  be a permutation.
  \begin{compactenum}[\rm (a)]
  \item \label{item:southeastgens} $I_w(\partial)$ is generated by the
    minors coming from the submatrices whose lower right corner is a
    southeast corner of $D(w)$.
  \item \label{item:schubertprime} $I_w(\partial)$ is a prime ideal of
    codimension $\ell(w)$.
  \item \label{item:normalCM} $X(w)$ is a normal variety.
  \end{compactenum}
\end{theorem}

\begin{proof} See \cite[Chapter 15]{cca} for
  \eqref{item:southeastgens} and \eqref{item:schubertprime}. For
  \eqref{item:normalCM}, we can realize $X(w)$ as a product of an
  affine space with an open subset of a Schubert variety in the
  complete flag variety (see Step 2 of the proof of
  Theorem~\ref{acyclicschubert} for more details), so it is enough to
  know that Schubert varieties are normal. This is proven in
  \cite[Theorem 3]{ramanan}.

  See also \cite[Theorem 2.4.3]{knutson} for more about the
  relationship of local properties for Schubert varieties and local
  properties of a product of matrix Schubert varieties with affine
  spaces.
\end{proof}

Given any scheme $X$, we let ${\rm K}(X)$ denote the K-theory of
coherent sheaves on $X$. This is the free Abelian group generated by
the symbols $[\FF]$ for each coherent sheaf $\FF$ modulo the relations
$[\FF] = [\FF'] + [\FF'']$ for each short exact sequence of the form
\[
0 \to \FF' \to \FF \to \FF'' \to 0.
\]
Given a finite complex $\CC_\bullet$ of coherent sheaves, we set
$[\CC_\bullet] = \sum_i (-1)^i [\CC_i] = \sum_i (-1)^i
[\Hs_i(\CC_\bullet)]$. If $X$ is nonsingular and finite-dimensional,
then ${\rm K}(X)$ has a ring structure given by 
\[ [\FF][\FF'] = \sum_{i=0}^{\dim X} (-1)^i [{\cal T}{\rm
  or}_i^{\OO_X}(\FF, \FF')].
\]

Now suppose that $X$ is an equidimensional smooth quasi-projective
variety over $K$. For $k \ge 0$, let $F^k {\rm K}(X)$ be the subgroup
of ${\rm K}(X)$ generated by coherent sheaves whose support has
codimension at least $k$, and set $\gr^k {\rm K}(X) = F^k {\rm K}(X) /
F^{k+1} {\rm K}(X)$. This filtration is compatible with the ring
structure on ${\rm K}(X)$ \cite[Th\'eor\`eme 2.12, Corollaire
1]{grothendieck}, and we set $\gr {\rm K}(X) = \bigoplus_{k \ge 0}
\gr^k {\rm K}(X)$ to be the associated graded ring.

Let ${\rm A}^*(X)$ be the Chow ring of $X$. We identify this with the
direct sum of Chow groups ${\rm A}_*(X)$ of $X$ via the isomorphism $c
\mapsto c \cap [X]$. Let $\phi \colon {\rm A}^*(X) \to \gr {\rm K}(X)$
be the functorial morphism of graded rings which for a closed
subvariety $V \subseteq X$ sends $[V]$ to $[\OO_V]$. If $\FF$ is a
coherent sheaf whose support has codimension at least $k$, then we
have $\phi(Z^k(\FF)) = [\FF]$ as elements of $\gr^k {\rm K}(X)$ where
\begin{align} \label{eqn:multiplicity}
Z^k(\FF) = \sum_{\codim V = k} m_V(\FF)[V],
\end{align}
and $m_V(\FF)$ is the length of the stalk of $\FF$ at the generic
point of $V$. We will need to know later that $\phi$ becomes an
isomorphism after tensoring with $\Q$. See \cite[Example 15.1.5,
15.2.16]{intersection} for more details. For a finite complex of
vector bundles $C_\bullet$ such that $[C_\bullet] \in F^k {\rm K}(X)$,
we use $[C_\bullet]_k$ to denote the corresponding element of $\gr^k
{\rm K}(X)$.

\begin{lemma} \label{lemma:schubcalc} The identity $\phi(\ss_w(x,y)) =
  [C_\bullet]_{\ell(w)}$ holds.
\end{lemma}

\begin{proof}
For a line bundle $L$ of the form $\OO(D)$ where $D$ is an irreducible
divisor, we have $c_1(L) \cap [X] = [D]$ \cite[Theorem
3.2(f)]{intersection}. Hence
\[
\phi(c_1(L) \cap [X]) = (1 - [L^\vee])_1 \in \gr^1 {\rm K}(X)
\]
by the short exact sequence
\[
0 \to \OO(-D) \to \OO_X \to \OO_D \to 0.
\]
So the same formula holds for all $L$ by linearity, and 
\[
\phi(x_i) = 1 - [E_i / E_{i-1}], \quad \phi(y_j) = 1 - [\ker(F_j \surj
F_{j-1})].
\]
Let $a$ and $b$ be a new set of variables. We have $\ss_w(a,b) =
\sum_{u \le_W w} \ss_u(a) \ss_{uw^{-1}}(-b)$ by
\eqref{induction}. Doing the transformation $a_i \mapsto x_i - 1$ and
$b_j \mapsto y_j - 1$, we get 
\[
\phi(\ss_w(a,b)) = \sum_{u \le_W w} (-1)^{\ell(u)} \ss_u(E)
\ss_{uw^{-1}}(F).
\]
By Proposition~\ref{complexterms}, this sum is $[\Schub_w(\partial)]$
(the change from $uw^{-1}$ to $wu^{-1}$ is a consequence of the fact
that $F_1$ in Proposition~\ref{complexterms} contains only odd
elements). So it is enough to show that the substitution $a_i \mapsto
a_i + 1$, $b_j \mapsto b_j + 1$ leaves the expression $\ss_w(a,b)$
invariant. This is clearly true for $\ss_{w_0}(x,y) = \prod_{i + j \le
  n} (x_i - y_j)$, and holds for an arbitrary permutation because the
divided difference operators (see \eqref{divideddiff}) applied to a
substitution invariant function yield a substitution invariant
function.
\end{proof}

The flag variety $\Flag(V)$ is smooth, and its K-theory is freely
generated as a group by the structure sheaves $[\OO_{X_w}]$ (see
\cite[Examples 1.9.1, 15.2.16]{intersection}). Also, the irreducible
components of any $B$-equivariant subvariety of $\Flag(V)$ must be
Schubert varieties. There is a tautological flag of subbundles
\[
0 = \RR_0 \subset \RR_1 \subset \RR_2 \subset \cdots \subset \RR_{n-1}
\subset \RR_n = V \times \Flag(V)
\]
on $\Flag(V)$, where the fiber of $\RR_i$ over a flag $W_\bullet$ is
the space $W_i$. Setting $x_i = -c_1(\RR_i / \RR_{i-1})$, the Schubert
polynomial $\ss_w(x_1, \dots, x_n)$ represents the Poincar\'e dual of
the Schubert variety $X_{w_0w}$ (see, for example, \cite[Theorem
3.6.18]{manivel}).

\begin{corollary} \label{corollary:ktheory} Let $V$ be an
  $n$-dimensional vector space and let $\CC$ be the Schubert complex
  associated with the permutation $w$ and the identity map of $V
  \times \Flag(V)$ to itself, where the subspace flag consists of the
  tautological subbundles and the quotient flag consists of trivial
  vector bundles. Then $[\CC]_{\ell(w)} = [\OO_{X_{w_0w}}]$ in
  $\gr^{\ell(w)} {\rm K}(\Flag(V))$.
\end{corollary}

\begin{proof} Both quantities agree with $\phi(\ss_w(x,0))$ where $x_i
  = -c_1(\RR_i / \RR_{i-1})$ for $i=1,\dots,n$.
\end{proof}

\subsection{Generic acyclicity of Schubert complexes.}

\begin{theorem} \label{acyclicschubert} Let $A = K[\partial_{i,j}]$ be
  a polynomial ring over a field $K$, and let $\partial \colon F_0 \to
  F_1$ be a generic map of variables between two free $A$-modules.
  \begin{compactenum}[\rm (a)]
  \item \label{acyclicitem} The Schubert complex $\Schub_w(\partial)$
    is acyclic, and resolves a Cohen--Macaulay module $M$ of
    codimension $\ell(w)$ supported in $I_{w^{-1}}(\partial) \subseteq
    A$.
  \item \label{localitem} The restriction of $M$ to $X(w^{-1})$ is a
    line bundle outside of a certain codimension 2 subset.
  \item \label{integeritem} The Schubert complex defined over the
    integers is acyclic.
  \end{compactenum}
\end{theorem}

Before we begin the proof, let us outline the strategy. The main idea
is to use the filtration given by Proposition~\ref{complexfiltration}
and work by induction. The main difficulty is the fact that there is a
homological shift in the filtration, which only allows one to conclude
that $\Hs_i(\Schub_w(\partial)) = 0$ for $i>1$ (see
\eqref{crucialsequence}). Hence the class of $C = \Schub_w(\partial)$
in an appropriate Grothendieck group is $[\Hs_0(C)] - [\Hs_1(C)]$. To
make this expression more useful, we work with a sheaf version $\CC$
of $C$ over a flag variety, where the K-theory possesses a nice
basis. To get a handle on $[\CC]$, we work with an associated graded
of K-theory and show that the top degree terms of $[\Hs_0(\CC)]$ and
$[\CC]$ agree. Finally, we show that the support of $\Hs_1(\CC)$ must
be a proper closed subset of the support of $\CC$, and we use this to
show that $\Hs_1(\CC)$ must be 0.

\begin{proof} We will prove the statement by induction first on
  $\ell(w)$ and second on the index of $w$ (see
  Section~\ref{section:filtration} for definitions). The case $w = 1$
  is immediate. Using the notation of
  Proposition~\ref{complexfiltration}, it is immediate that $C'$ is
  acyclic by induction and the long exact sequence on homology. Hence
  we only need to analyze the short exact sequence
  \[
  0 \to C' \to C \to \Schub_v(\partial)[-1] \otimes
  F_0^\alpha/F_0^{\alpha-1} \to 0.
  \]
  The induced long exact sequence is
  \begin{align} \label{crucialsequence}
  0 \to \Hs_1(C) \to \Hs_0(\Schub_v(\partial)) \otimes F_0^\alpha /
  F_0^{\alpha-1} \to \Hs_0(C') \to \Hs_0(C) \to 0,
  \end{align}
  so we have to show that $\Hs_1(C) = 0$, and that the support of
  $\Hs_0(C) = M$ is $P = I_{w^{-1}}(\partial)$. We proceed in steps.

  ~

  \noindent {\bf Step 1.} We first show that the length of
  $\Hs_0(C)_P$ restricted to $X(w^{-1})$ is at most 1.

  ~
  
  The short exact sequence
  \[
  0 \to C_k \to C' \to \Schub_v(\partial) \otimes \langle
  e'_{w(\beta)} \rangle \to 0
  \]
  induces the sequence
  \[
  0 \to \Hs_0(C_k) \to \Hs_0(C') \to \Hs_0(\Schub_v(\partial)) \otimes
  \langle e'_{w(\beta)} \rangle \to 0.
  \]
  By induction on the filtration in Proposition~\ref{complexfiltration},
  the support of $\Hs_0(C_k)$ is in the union of the $X(\psi_t^{-1})$,
  and hence does not contain $X(w^{-1})$. So localizing at $P$, we
  get an isomorphism
  \[
  \Hs_0(C')_P \cong \Hs_0(\Schub_v(\partial))_P \otimes \langle
  e'_{w(\beta)} \rangle.
  \]
  So we can restrict this isomorphism to $X(w^{-1})$. Localizing
  \eqref{crucialsequence} at $P$ and then restricting to $X(w^{-1})$,
  we get a surjection
  \[
  \Hs_0(\Schub_v(\partial))_P \otimes \langle e'_{w(\beta)} \rangle
  \to \Hs_0(C)_P \to 0.
  \]
  By induction, \eqref{localitem} gives that the first term has length
  1 over the generic point of $X(w^{-1})$, so $\length(\Hs_0(C)_P) \le
  1$.

  ~

  \noindent {\bf Step 2.} We show that the length of $\Hs_0(C)_P$
  restricted to $X(w^{-1})$ is exactly 1.

  ~

  Without loss of generality, we may extend $\partial$ to a generic
  $2n \times 2n$ matrix by embedding it in the upper left
  corner. Since $w \in \Sigma_n$, these new variables do not affect
  the Schubert complex when we interpret $w$ as a permutation of
  $\Sigma_{2n}$ by having it fix $\{n+1,\dots,2n\}$, so we will refer
  to them as {\bf irrelevant variables}. Now consider the Schubert
  complex $\CC$ on the complete flag variety $Z$ of a vector space of
  dimension $2n$, where the even flag is given by the tautological
  flag of vector bundles $\RR_1 \subset \RR_2 \subset \cdots \subset
  \RR_n$, and the odd flag is given by the trivial vector bundles $V_i
  = \langle e_1, \dots, e_i \rangle$. We identify $Z$ with a quotient
  $\GL(V_{2n})/B$. Restricting to the unique open $B$-orbit $\Omega =
  \Omega_{w_0}$ of $Z$ (which is an affine space), we return to the
  current situation with some of the irrelevant variables of
  $\partial$ specialized to 1 and some specialized to 0. So to finish
  this step, we only need to show that $\length(\Hs_0(\CC)_P) = 1$.

  Let $w_0 \in \Sigma_{2n}$ be the long word. Identify $V_i$ over
  $\Omega$ with $F_1^{-2n-1+i}$. Then the intersection $X_{w_0w} \cap
  \Omega$ is defined by the ideal $I_{w^{-1}}(\partial)$:
  \begin{align*}
    \dim(W_p \cap V_q) \ge r_{w_0w}(p,q) &\iff \dim(W_p \cap
    F_1^{-2n-1+q}) \ge r_{w_0w}(p,q)\\
    &\iff \rank(W_p \to F_1/F_1^{-2n-1+q}) \le p - r_{w_0w}(p,q)\\
    &\iff \rank(W_p \to F_1/F_1^{-1-q}) \le p - r_{w_0w}(p,2n-q) =
    r_w(p,q),
  \end{align*}
  and the map $W_p \to F_1/F_1^{-1-q}$ is given by the upper left $q
  \times p$ submatrix of $\partial$. 

  From the earlier discussion, $[\CC] = [\Hs_0(\CC)] -
  [\Hs_1(\CC)]$. By Corollary~\ref{corollary:ktheory}, the top
  dimension term of $[\CC]$ is $[\OO_{X_{w_0w}}]$. So
  $\length(\Hs_0(\CC)_P) - \length(\Hs_1(\CC)_P) = 1$. We showed above
  that the first length is at most 1, which means that it must be 1,
  and the stalk of $\Hs_1(\CC)$ at the generic point of $X_{w_0w}$
  must be 0.

  ~

  \noindent {\bf Step 3.} The annihilator of $\Hs_0(C)$ properly
  contains $I_{v^{-1}}(\partial)$.

  ~

  We have that $D(w) = D(v) \cup \{(\alpha,w(\beta))\}$, and
  $(\alpha,w(\beta))$ is a southeast corner of $D(w)$: no boxes of
  $D(w)$ lie directly below or to the right of it. This means in
  particular that $I_{w^{-1}}$ is generated by $I_{v^{-1}}$ and the
  $(r+1) \times (r+1)$ minors of the upper $w(\beta) \times \alpha$
  submatrix of $\partial$, where $r = r_w(\alpha,w(\beta))$. We will
  show that a minor in $I_{w^{-1}}$ which is not in $I_{v^{-1}}$
  annihilates $\Hs_0(C)$.

  The module $\Hs_0(C)$ is generated by the BSLs of $D(w)$ that only
  contain marked letters. We have that $w(\beta) - r$ is the number of
  boxes in $D(w)$ in the $\alpha$th row. Let $J = \{ w(\beta) - r,
  \dots, w(\beta) \}$ and let $I$ be an $(r+1)$-subset of $\{1, \dots,
  \alpha\}$. Set $M_{J,I}$ to be the minor of $\partial$ consisting of
  the rows indexed by $J$ and the columns indexed by $I$. We will show
  that $M_{J,I}$ annihilates $\Hs_0(C)$.

  Given a label $j$, and a labeling $T$ of the first $\alpha-1$ rows
  of $D(w)$, let $T(j)$ be the labeling of $D(w)$ that agrees with $T$
  for the first $\alpha-1$ rows, and in which the $i$th box in the
  $\alpha$th row (going from left to right) has the label $i'$ for
  $i=1,\dots, w(\beta)-r-1$, and the box $(\alpha, w(\beta))$ has the
  label $j$. Let $d \colon C_1 \to C_0$ denote the differential. Then
  $d(T(j)) = \sum_{k=1}^{w(\beta)} \partial_{k,j} T(k')$. Note that
  $T(k') = 0$ whenever $1 \le k < w(\beta)-r$ since in this case the
  label $k'$ appears in the bottom row twice. Since $\alpha > r$, the
  $\alpha$ equations
  \begin{align*} 
    \sum_{k = w(\beta)-r}^{w(\beta)} \partial_{k,j} T(k') = 0 \text{
      for } j=1,\dots,\alpha
  \end{align*}
  in $\Hs_0(C)$ show that $M_{J,I}$ annihilates $T(k')$ for $1 \le k
  \le w(\beta)$.

  It remains to show that $M_{J,I}$ annihilates the elements $T$ where
  the labels in the first $w(\beta) - r - 1$ boxes of the $\alpha$th
  row of $T$ are allowed to take values in $\{(w(\beta)-r)', \dots,
  w(\beta)'\}$. It is enough to show how to vary the entries one box
  at a time by decreasing their values (remembering that $i' < j'$ if
  $i > j$). So fix a column index $c$ which contains the $i$th box in
  row $\alpha$ and choose $j > i$. Let $T_j$ denote the labeling
  obtained from $T$ by changing the label in $(\alpha, c)$ from $i'$
  to $j'$. Let $X \in \b$ be the matrix which sends the basis vector
  $e'_i$ to $e'_j$ and kills all other basis vectors. Then $X \cdot T$
  is equal to $T_j$ plus other terms whose labels in the $\alpha$th
  row are the same as those of $T$, and hence are annihilated by
  $M_{J,I}$. Since the actions of $\b$ and $A$ commute with one
  another, we conclude that $M_{J,I}$ annihilates $T_j$.

  ~

  \noindent {\bf Step 4.} We show that $\Hs_1(C) = 0$.

  ~

  By examining different open affine charts of $Z$, Step 3 shows that
  the support of $\Hs_0(\CC)$ is a proper subset of $X_{w_0v}$. The
  argument in Step 2 implies that the same is true for $\Hs_1(\CC)$
  since the structure sheaves of the Schubert varieties form a basis
  for ${\rm K}(Z)$. So the codimension of the support of $\Hs_1(\CC)$
  is at least $\ell(w)$. Name the differentials in the complex $d_i
  \colon \CC_i \to \CC_{i-1}$. Restrict to an open affine set. Let
  $r_i$ be the rank of $d_i$, and set $I(d_i)$ to be the ideal
  generated by the $r_i \times r_i$ minors of $d_i$. Let $Q$ be a
  prime ideal which does not contain $\sqrt{I(d_1)} =
  \ann(\Hs_0(\CC))$. Then $\Hs_0(\CC)_Q = 0$, which makes the
  localization $(d_1)_Q \colon (\CC_1)_Q \to (\CC_0)_Q$ a split
  surjection. Let $\CC'_1$ be the quotient of a splitting of
  $(d_1)_Q$, so that we have a free resolution
  \[
  0 \to (\CC_{\ell(w)})_Q \to \cdots \to (\CC_2)_Q \to \CC'_1
  \]
  of $\Hs_1(\CC)_Q$. Hence the projective dimension of $\Hs_1(\CC)_Q$
  is at most $\ell(w)-1$. In general, localizing can only increase the
  codimension of a module (if we interpret the codimension of 0 to be
  $\infty$), so $\codim \Hs_1(\CC)_Q \ge \ell(w)$. This is also equal
  to the depth of its annihilator since $Z$ is nonsingular. So the
  inequality 
  \[
  \operatorname{proj. dim} \Hs_1(\CC)_Q < \depth \ann \Hs_1(\CC)_Q
  \]
  contradicts \cite[Corollary 18.5]{eisenbud} unless $\Hs_1(\CC)_Q =
  0$. This implies that $\sqrt{I(d_2)_Q}$ is the unit ideal, which
  means that $\sqrt{I(d_2)} \nsubseteq Q$. Hence we conclude that any
  prime ideal which contains $\sqrt{I(d_2)}$ also contains
  $\sqrt{I(d_1)}$. Since a radical ideal is the intersection of the
  prime ideals containing it, we conclude that $\sqrt{I(d_1)}
  \subseteq \sqrt{I(d_2)}$. We also get the inclusions
  \[
  \sqrt{I(d_2)} \subseteq \sqrt{I(d_3)} \subseteq \cdots
  \subseteq \sqrt{I(d_\ell(w))}
  \]
  since the rest of the homology of $\CC$ vanishes \cite[Corollary
  20.12]{eisenbud}, so $\depth I(d_i) \ge \depth I(d_1) \ge \ell(w)$
  for all $i$. We conclude the acyclicity of $\CC$ using the
  Buchsbaum--Eisenbud criterion \cite[Theorem 20.9]{eisenbud} (the
  complex is acyclic at the generic point, and the rank of a map over
  an integral domain stays the same upon passing to its field of
  fractions, so the rank conditions of this criterion are satisfied).

  ~

  \noindent {\bf Step 5.} We show that the restriction of $M =
  \Hs_0(C)$ to $X(w^{-1})$ is a line bundle, and that its support is
  exactly $X(w^{-1})$.

  ~

  Since the projective dimension of $M$ is 1 more than the projective
  dimension of $\Hs_0(\Schub_v(\partial))$, the codimension of its
  support can increase by at most 1 by the Auslander--Buchsbaum
  formula \cite[Theorem 19.9]{eisenbud}.  Thus if we can show that $P$
  is contained in the annihilator of $M$, then it must be equal to its
  annihilator. We have already done this by showing that the stalk of
  $M$ at the generic point of $X(w^{-1})$ is nonzero. Thus the
  codimension and projective dimension of $M$ coincide, which means
  that it is Cohen--Macaulay by the Auslander--Buchsbaum formula. So
  \eqref{acyclicitem} is proven.

  Let $Q$ be the prime ideal associated with a codimension 1
  subvariety of $X(w^{-1})$. To prove \eqref{localitem}, we only need
  to show that $M_Q$ is generated by 1 element. Since $X(w^{-1})$ is
  normal (Theorem \ref{theorem:matrixschubert}\eqref{item:normalCM}),
  the local ring $R = \OO_{X(v^{-1}),Q}$ is a discrete valuation ring,
  and hence regular. Furthermore, we have established already that $M$
  is Cohen--Macaulay, so $M_Q$ is a free $R$-module by the
  Auslander--Buchsbaum formula. So $M$ is free in some open
  neighborhood around $Q$. Since $M$ is generated by a single element
  generically (after further localization), we conclude that $M_Q$
  must also be generated by 1 element.

  Now \eqref{integeritem} follows since we have shown acyclicity over
  an arbitrary field.
\end{proof}

\begin{corollary} \label{acyclicgeometric} Let $X$ be an
  equidimensional Cohen--Macaulay scheme, and let $\partial \colon E \to
  F$ be a map of vector bundles on $X$. Let $E_1 \subset \cdots
  \subset E_n = E$ and $F^{-n} \subset \cdots \subset F^{-1} = F$ be
  split flags of subbundles. Let $w \in \Sigma_n$ be a permutation,
  and define the degeneracy locus
  \[
  D_w(\partial) = \{ x \in X \mid \rank(\partial_x \colon E_p(x) \to
  F/F^{-q-1}(x)) \le r_w(p,q) \},
  \] 
  where the ideal sheaf of $D_w(\partial)$ is locally generated by the
  minors given by the rank conditions. Suppose that $D_w(\partial)$
  has codimension $\ell(w)$.
  \begin{compactenum}[\rm (a)]
  \item \label{acyclicitem} The Schubert complex $\Schub_w(\partial)$
    is acyclic, and the support of its cokernel $\LL$ is
    $D_w(\partial)$.
  \item \label{CMitem} The degeneracy locus $D_w(\partial)$ is
    Cohen--Macaulay.
  \item \label{linebundleitem} The restriction of $\LL$ to
    $D_w(\partial)$ is a line bundle outside of a certain codimension
    2 subset.
  \end{compactenum}
\end{corollary}

\begin{proof} The statement is local, so we can replace $X$ by $\Spec
  R$ where $R$ is a local Cohen--Macaulay ring. In this case,
  $D_w(\partial)$ is defined by the ideal $I_{w^{-1}}(\partial)$. Let
  $\partial^g$ denote the generic matrix, and let $(C_\bullet,
  d_\bullet)$ be the complex over $\Z[\partial^g_{i,j}]$ as in
  Theorem~\ref{acyclicschubert}. We get $(C'_\bullet, d'_\bullet) =
  \Schub_w(\partial)$ by specializing the variables $\partial^g_{i,j}$
  to elements of $R$ and base changing to $R$. Let $r_i$ be the rank
  of $d_i$, and let $I(d_i)$ be the ideal generated by the $r_i \times
  r_i$ minors of $d_i$. By \cite[Corollary 20.12]{eisenbud},
  $\sqrt{I(d_1)} = \sqrt{I(d_2)} = \cdots = \sqrt{I(d_{\ell(w)})}$
  since $C$ is acyclic and since $\depth I_{w^{-1}}(\partial^g) =
  \ell(w)$.

  Specializing $\partial$ to elements of $R$, the same equalities hold
  when replacing $d_i$ with $d'_i$. Noting that $I(d'_1) = \ann \coker
  d'_1 \supseteq I_{w^{-1}}(\partial)$, we get that
  \[
  \depth I(d'_1) \ge \depth I_{w^{-1}}(\partial) = \codim
  I_{w^{-1}}(\partial) = \ell(w)
  \]
  by the assumptions that $R$ is Cohen--Macaulay and that
  $D_w(\partial)$ has codimension $\ell(w)$. Hence $\depth I(d'_i) \ge
  \ell(w)$ for $i=1,\dots,\ell(w)$, which means that $C'$ is acyclic
  by the Buchsbaum--Eisenbud acyclicity criterion \cite[Theorem
  20.9]{eisenbud}. Finally, since the length of the Schubert complex
  is $\ell(w)$, we conclude that the depth of the cokernel must be
  $\ell(w)$ by the Auslander--Buchsbaum formula. So in fact $\ann
  \coker d'_1 = I_{w^{-1}}(\partial)$, which implies that the support
  of the cokernel is $D_w(\partial)$. This establishes
  \eqref{acyclicitem} and \eqref{CMitem}.

  Now \eqref{linebundleitem} follows from
  Theorem~\ref{acyclicschubert}\eqref{localitem}.
\end{proof}

\subsection{Examples.}

\begin{example} \label{eg:simpletranspositions} Let $s_i$ denote the
  simple transposition that switches $i$ and $i+1$. For $w = s_i$, the
  maximal transition \eqref{maxtransition} simplifies to $(\alpha,
  \beta) = (i, i+1)$, $v = 1$, $k=1$, and $\psi_1 = s_{i-1}$. This is
  also evident from the fact that $\ss_{s_i}(x,y) = x_1 + \cdots + x_i
  - y_1 - \cdots - y_i$.

  Let $F_0$ and $F_1$ be vector spaces of dimension $n$ with $n \ge
  i$. Given a map $\partial \colon F_0 \to F_1$ with distinguished
  bases $e_1, \dots, e_n$ and $e'_1, \dots, e'_n$ (coming from a flag
  of $F_0$ and a quotient flag of $F_1$), respectively, the associated
  Schubert complex $\Schub_{s_i}(\partial)$ is obtained from
  $\partial$ by taking the upper left $i \times i$ submatrix of
  $\partial$.

  The filtration of Proposition~\ref{complexfiltration} can be described
  as follows. First, it should look like
  \[
  0 = C_0 \subset C_1 \subset C' \subset C = \Schub_{s_i}(\partial)
  \]
  where $C_1 \cong \Schub_{s_{i-1}}(\partial)$, $C' / C_1 \cong
  F_1^{-i} / F_1^{-i+1}$ and $C / C' \cong (F_0^i /
  F_0^{i-1})[-1]$. 

  Then $C'$ is the subcomplex $\langle e_1, \dots, e_{i-1} \rangle \to
  \langle e'_1, \dots, e'_i \rangle$ of $C$, so the quotient is then
  $F_0^i / F_0^{i-1}$ concentrated in degree 1.  Finally, $C_1$ is the
  subcomplex $\langle e_1, \dots, e_{i-1} \rangle \to \langle e'_1,
  \dots, e'_{i-1} \rangle$ which is isomorphic to
  $\Schub_{s_{i-1}}(\partial)$ and the quotient $C' / C_1$ is
  $F_1^{-i} / F_1^{-i+1}$ as required.
\end{example}

Here is a combinatorial description of the differentials in the
Schubert complex for a flagged isomorphism. We will work with just the
tensor product complex $\bigotimes_{k=1}^{n-1} \D^{r_k(w)}(F)$. Then
the basis elements of its terms are row-strict labelings. The
differential sends such a labeling to the signed sum of all possible
ways to change a single unmarked letter to the corresponding marked
letter. If $T'$ is obtained from $T$ by marking a letter in the $i$th
row, then the sign on $T'$ is $(-1)^n$, where $n$ is the number of
unmarked letters of $T$ in the first $i-1$ rows.

\begin{example} Consider the permutation $w = 1423$. Then $D(w) = \{(2,2),
  (2,3)\}$, and if we use the identity matrix $I$, $\Schub_w(I)$ looks
  like
  \[
  \begin{array}{lll}
    { \tableau[scY]{2,2} }
    \mapsto 
    { \tableau[scY]{2, 2'} }, &
    { \tableau[scY]{2, 2'} } \mapsto 0, &
    { \tableau[scY]{1', 2'} } \mapsto 0,  \\ \\
    { \tableau[scY]{2, 1} } \mapsto { \tableau[scY]{1, 2'} } + {
      \tableau[scY]{2, 1'} }, 
    &  { \tableau[scY]{1, 2'} } \mapsto { \tableau[scY]{1',2'} }, &
    { \tableau[scY]{1', 3'} } \mapsto 0, \\ \\
    { \tableau[scY]{1, 1} } \mapsto 
    { \tableau[scY]{1,1'} }, &
    { \tableau[scY]{2, 1'} } \mapsto -{ \tableau[scY]{1',2'} }, &
    { \tableau[scY]{2', 3'} } \mapsto 0. \\ \\
    &  { \tableau[scY]{1,1'} } \mapsto 0, \\ \\ 
    &  { \tableau[scY]{1,3'} } \mapsto { \tableau[scY]{1',3'} },
    \\ \\ 
    & { \tableau[scY]{2,3'} } \mapsto { \tableau[scY]{2',3'} }, \\ \\
  \end{array}
  \]
  If we use a generic map $e_1 \mapsto ae'_1 + be'_2 + ce'_3$ and $e_2
  \mapsto de'_1 + ee'_2 + fe'_3$ (the images of $e_3$ and $e_4$ are
  irrelevant, and it is also irrelevant to map to $e'_4$) instead,
  then the complex looks like
  \[
  0 \to A^3
  \xrightarrow{\begin{pmatrix}e& b&      0\\
      0& e&      b\\
      d& a&      0\\
      0& d&      a\\
      0& f&      c\\
      f& c&      0
    \end{pmatrix}} A^6
  \xrightarrow{\begin{pmatrix}d& a& {-e}& {-b}& 0&      0\\
      0& 0& {-f}& {-c}& a&      d\\
      {-f}& {-c}& 0& 0& b& e
     \end{pmatrix}}
   A^3 \to M \to 0
  \]
  The cokernel $M$ is Cohen--Macaulay of codimension 2 over $A =
  K[a,b,c,d,e,f]$.
\end{example}

\begin{example} Consider the permutation $w = 2413$. Then $D(w) = \{(1,1),
  (2,1), (2,3) \}$, and if we use the identity matrix $I$,
  $\Schub_w(I)$ looks like (note the negative signs which come from
  the fact that we are working with an image of a tensor product of
  two divided power complexes)
  \[
  \begin{array}{lll} 
    {\tableau[scY]{1|2,\bl ,2}} \mapsto
    {\tableau[scY]{1'|2,\bl ,2}} - 
    {\tableau[scY]{1|2,\bl ,2'}}, & {\tableau[scY]{1|2,\bl ,1}}
    \mapsto {\tableau[scY]{1'|2,\bl ,1}} - {\tableau[scY]{1|2,\bl
        ,1'}}\\ \\
    {\tableau[scY]{1|2,\bl ,1'}} \mapsto {\tableau[scY]{1'|2,\bl
        ,1'}} + {\tableau[scY]{1'|1,\bl ,2'}}, &
    {\tableau[scY]{1|2,\bl ,2'}} \mapsto {\tableau[scY]{1'|2,\bl
        ,2'}}, & {\tableau[scY]{1|2,\bl ,3'}} \mapsto
    {\tableau[scY]{1'|2,\bl ,3'}}\\ \\
    {\tableau[scY]{1'|1,\bl ,1}} \mapsto 
    {\tableau[scY]{1'|1,\bl ,1'}}, & {\tableau[scY]{1'|2,\bl ,2}}
    \mapsto 
    {\tableau[scY]{1'|2,\bl ,2'}}, & {\tableau[scY]{1'|2,\bl ,1}}
    \mapsto {\tableau[scY]{1'|1,\bl ,2'}} +
    {\tableau[scY]{1'|2,\bl ,1'}}
  \end{array}
  \]
  \[
  \begin{array}{lll}
    {\tableau[scY]{1'|1,\bl ,1'}} \mapsto 0, &
    {\tableau[scY]{1'|1,\bl ,2'}} \mapsto
    {\tableau[scY]{1'|1',\bl ,2'}}, &
    {\tableau[scY]{1'|1,\bl ,3'}} \mapsto
    {\tableau[scY]{1'|1',\bl ,3'}}\\ \\
    {\tableau[scY]{1'|2,\bl ,1'}} \mapsto
    -{\tableau[scY]{1'|1', \bl ,2'}}, & {\tableau[scY]{1'|2,\bl
        ,2'}} \mapsto 0, & 
    {\tableau[scY]{1'|2,\bl ,3'}} \mapsto 0\\ \\
    {\tableau[scY]{1'|1',\bl ,2'}} \mapsto 0, &
    {\tableau[scY]{1'|1',\bl ,3'}} \mapsto 0
  \end{array}
  \]
  Using a generic matrix defined by $e_1 \mapsto ae'_1 + be'_2 +
  ce'_3$ and $e_2 \mapsto de'_1 + ee'_2 + fe'_3$ (the other
  coefficients are irrelevant) instead of the identity matrix gives
  the following complex
  \[
  0 \to A^2 \xrightarrow{\begin{pmatrix}{-d}&    {-a}\\
      {-e}&    {-b}\\
      {-f}&    {-c}\\
      0&    {-d}\\
      a&    0\\
      0& a
    \end{pmatrix}} A^6
  \xrightarrow{\begin{pmatrix}0& 0& 0& a& 0&    d\\
      e& {-d}& 0& b& 0&    e\\
      f& 0& {-d}& c& 0&    f\\
      a& 0& 0& 0& d&    a\\
      0& a& 0& 0& e&    b\\
      0& 0& a& 0& f& c
    \end{pmatrix}} A^6
  \xrightarrow{\begin{pmatrix}{-b}& a& 0& {-e}& d& 0\\
      {-c}& 0& a& {-f}& 0& d
    \end{pmatrix}} A^2 \to M \to 0
  \]
  Its cokernel $M$ is Cohen--Macaulay of codimension 3 over $A =
  K[a,b,c,d,e,f]$.
\end{example}

\section{Degeneracy loci.} \label{locisection}

\subsection{A formula of Fulton.} \label{fultonsection}

Suppose we are given a map $\partial \colon E \to F$ of vector bundles
of rank $n$ on a scheme $X$, together with a flag of subbundles $E_1
\subset E_2 \subset \cdots \subset E_n = E$ and a flag of quotient
bundles $F = F_n \surj F_{n-1} \surj \cdots \surj F_1$ such that
$\rank E_i = \rank F_i = i$. We assume that the quotients $E_i /
E_{i+1}$ are locally free. For a permutation $w$, define
\[
D_w(\partial) = \{ x \in X \mid \rank (\partial_x \colon E_p(x) \to
F_q(x)) \le r_w(p,q) \}.
\]
Then $\codim D_w(\partial) \le \ell(w)$. Define Chern classes $x_i =
-c_1(E_i / E_{i-1})$ and $y_i = -c_1(\ker(F_i \surj F_{i-1}))$.

\begin{theorem}[Fulton] \label{fultonformula} Suppose that $X$ is an
  equidimensional Cohen--Macaulay scheme of finite type over a field
  $K$ and $D_w(\partial)$ has codimension $\ell(w)$. Then the identity
  \begin{align*} 
  [D_w(\partial)] = \ss_w(x,y) \cap [X]
  \end{align*}
  holds in the Chow group $\mathrm{A}_{\dim(D_w(\partial))}(X)$.
\end{theorem}

See \cite[\S 8]{flags} for a more general statement which does not
enforce a codimension requirement on $D_w(\partial)$ or assume that
$X$ is Cohen--Macaulay. In order to state the connection between the
Schubert complex and Fulton's formula, we will need the following
lemma which was observed in \cite[Appendix 6]{pragacz}.

\begin{lemma} \label{lemma:pragacz} Let $X$ be an equidimensional
  smooth scheme of finite type over a field $K$, and let $D$ be an
  irreducible closed subscheme of $X$ of codimension $k$. Let
  $C_\bullet$ be a finite complex of vector bundles on $X$ and let
  $\alpha \in {\rm A}^{k}(X)$. If
  \[
  \supp C_\bullet = X \setminus \{ x \in X \mid (C_\bullet)|_x {\rm\
    is\ an\ exact\ complex} \}
  \]
  is contained in $D$, and $\phi(\alpha) = [C_\bullet]_{k}$, then
  $c[D] = \alpha$ for some $c \in \Q$.
\end{lemma}

For completeness (and since we have changed notation from
\cite{pragacz}), we will reproduce the proof.

\begin{proof} Let $i \colon D \to X$ and $j \colon X \setminus D \to
  X$ be the inclusions. Let the subscript $(-)_\Q$ denote tensoring
  with $\Q$. Then the diagram (of Abelian groups)
  \[
  \xymatrix{ {\rm A}^*(D)_\Q \ar[r]^-{i_*} & {\rm A}^*(X)_\Q
    \ar[r]^-{j_{\rm A}^*} \ar[d]^-{\phi_\Q} & {\rm A}^*(X \setminus
    D)_\Q \ar[r] \ar[d]^-{\phi_\Q} & 0 \\
    & \gr {\rm K}(X)_\Q \ar[r]^-{j_{\rm K}^*} & \gr {\rm K}(X \setminus
    D)_\Q }
  \]
  commutes by functoriality of $\phi$, and the first row is exact
  \cite[Proposition 1.8]{intersection}. Since $\supp(C_\bullet)
  \subseteq D$, we have $j_{\rm K}^*([C_\bullet]) = 0$, and since
  $\phi_\Q(\alpha) = [C_\bullet]$, we conclude that $j_{\rm
    A}^*(\alpha) = 0$ because $\phi_\Q$ is an isomorphism
  \cite[Example 15.2.16(b)]{intersection}. Since we assumed that
  $\alpha \in {\rm A}^{k}(X)$ we have $\alpha = i_*(\beta)$ for some
  $\beta \in {\rm A}^0(D)_\Q$. But $D$ is irreducible, and hence
  $\beta$ is some rational multiple of $[D]$. 
\end{proof}

We will verify Theorem~\ref{fultonformula} in the case that $X$ is a
smooth quasi-projective variety. The general case can be reduced to
this case using a ``universal construction'' and Chow's lemma (see
\cite[\S 8]{flags}).

\begin{proof}[Proof of Theorem~\ref{fultonformula}] 
  We will use Lemma~\ref{lemma:pragacz} with $D = D_w(\partial)$,
  $C_\bullet = \Schub_w(\partial)$, and $\alpha = \ss_w(x,y)$ using
  the notation from the beginning of this section. We know that $\supp
  C_\bullet \subseteq D$ and that the codimension of $D$ is $\ell(w) =
  \deg \alpha$ by Corollary~\ref{acyclicgeometric}. So in order to
  conclude Theorem~\ref{fultonformula}, we need to check that
  $\phi(\alpha) = [C_\bullet]$, which is the content of
  Lemma~\ref{lemma:schubcalc}. Finally, it remains to show that the
  constant given by Lemma~\ref{lemma:pragacz} is 1. This follows from
  \eqref{eqn:multiplicity} and
  Corollary~\ref{acyclicgeometric}\eqref{linebundleitem}.
\end{proof}

\subsection{Some remarks.}

\begin{remark} 
  The previous constructions do not require that the flags be
  complete, so that one can omit certain subbundles or quotient
  bundles as desired. The appropriate generalization would be to use
  partial flag varieties, but we have omitted such generality to keep
  the notation simpler.
\end{remark}

\begin{remark}
  A permutation $w \in \Sigma_n$ is {\bf Grassmannian} if it has at
  most one descent, i.e., there exists $r$ such that $w(1) < w(2) <
  \cdots < w(r) > w(r+1) < \cdots < w(n)$. Suppose that $w$ is {\bf
    bigrassmannian}, which means that $w$ and $w^{-1}$ are
  Grassmannian permutations. This is equivalent to saying that $D(w)$
  is a rectangle. In this case, the double Schubert polynomial
  $\ss_w(x,y)$ is a multi-Schur function for the partition $D(w)$ (one
  can use \cite[Proposition 2.6.8]{manivel} combined with
  \eqref{induction}). The degeneracy locus $D_w(\partial)$ can then be
  described by a single rank condition of the map $\partial \colon E
  \to F$, so the degeneracy locus formula of Fulton specializes to the
  Thom--Porteous formula mentioned in the introduction. So in
  principle, the action of $\b$ on $\Schub_w(\partial)$ should extend
  to an action of a general linear superalgebra, but it is not clear
  why this should be true without appealing to Schur polynomials.
\end{remark}

\begin{remark} The Schubert complex only gives a formula for the
  structure sheaf of the given degeneracy locus in the associated
  graded of K-theory. A formula for the structure sheaf in the actual
  K-theory is given in \cite[Theorem 3]{lascoux} using the so-called
  Grothendieck polynomials, but the formula is not obtained by
  constructing a complex, so it would be interesting to try to
  construct these complexes. The degeneracy loci for bigrassmannian
  permutations are determinantal varieties, and the resolutions in
  characteristic 0 are explained in \cite[\S 6.1]{weyman}. We should
  point out that the terms of the resolutions may change with the
  characteristic, see \cite[\S 6.2]{weyman}.
\end{remark}

\begin{remark}
  We have seen that the modules which are the cokernels of generic
  Schubert complexes have linear minimal free resolutions. These
  modules can be thought of as a ``linear approximation'' to the ideal
  which defines the matrix Schubert varieties, which in general have
  rich and complicated minimal free resolutions. More precisely, we
  have shown that matrix Schubert varieties possess maximal
  Cohen--Macaulay modules with linear resolutions. In general, the
  question of whether or not every graded ring possesses such a module
  is open (see \cite[p.543]{chowforms} for further information).

  Furthermore, such modules can be obtained geometrically, as outlined
  in \cite[Chapter 6, Exercises 34--36]{weyman} for the case of
  generic determinantal varieties and their symmetric and
  skew-symmetric analogues, which we will denote by $D \subset {\bf
    A}^N$. The idea is to find a projective variety $V$ and a
  subbundle $Z \subset V \times {\bf A}^N$ such that the projection $Z
  \to D$ is a desingularization. In each case, one can find a vector
  bundle on $Z$ whose pushforward to ${\bf A}^N$ provides the desired
  module supported on $D$. The proof that its minimal free resolution
  is linear involves some sheaf cohomology calculations.  It would be
  interesting to try to do this for matrix Schubert varieties, which
  are our affine models of Fulton's degeneracy loci. The
  desingularizations of matrix Schubert varieties one might try to use
  could be given by some analogue of Bott--Samelson varieties. The
  problem would then be to find the appropriate vector bundle and do
  the relevant sheaf cohomology calculations. It is the latter part
  that seems to be complicated.
\end{remark}

\small

\noindent Steven V Sam,
Department of Mathematics, 
Massachusetts Institute of Technology, 
Cambridge, MA 02139,\\
{\tt ssam@math.mit.edu}, 
\url{http://math.mit.edu/~ssam/}

\end{document}